%% file: improved_bounds_170523.tex
\documentclass[DIV=classic,a4paper]{artclcomp}

\input{head}

\usepackage{lipsum}
\usepackage{todonotes}
\usepackage{graphicx}
\usepackage{amsfonts}
\usepackage{accents}
\usepackage{xspace}
\usepackage{float}
\usepackage{dsfont}

\allowdisplaybreaks[1]

\def\set{\ensuremath{\mathcal{S}}}
\def\bx{\ensuremath{\mathbf{x}}}
\def\bu{\ensuremath{\mathbf{u}}}
\def\bv{\ensuremath{\mathbf{v}}}
\def\bs{\ensuremath{\mathbf{s}}}

\def\FH{Fr\'echet--Hoeffding\xspace}
\def\I{\ensuremath{\mathbb{I}}}
\def\ud{\mathrm{d}}
\def\b1{\mathds{1}}

\DeclareMathAccent{\what}{\mathord}{largesymbols}{"62}

\DeclareFontFamily{U}{mathx}{\hyphenchar\font45}
\DeclareFontShape{U}{mathx}{m}{n}{
      <5> <6> <7> <8> <9> <10>
      <10.95> <12> <14.4> <17.28> <20.74> <24.88>
      mathx10
      }{}
\DeclareSymbolFont{mathx}{U}{mathx}{m}{n}
\DeclareFontSubstitution{U}{mathx}{m}{n}
\DeclareMathAccent{\widecheck}{0}{mathx}{"71}

\begin{document}

\title{Improved Fr\'echet--Hoeffding bounds on $d$-copulas and applications in model-free finance}

\author[a,1,t]{Thibaut Lux}
\author[a,2,t]{Antonis Papapantoleon}

\address[a]{\emph{ Institute of Mathematics, TU Berlin, Stra{\ss}e des 17. Juni 136, 10623 Berlin, Germany}}

\eMail[1]{thibaut.lux@mail.tu-berlin.de}
\eMail[2]{papapan@math.tu-berlin.de}

\myThanks[t]{\emph{ We thank Peter Bank, Carole Bernard, Alfred M\"uller, Ludger R\"uschendorf and Peter Tankov for useful discussions during the work on these topics. 
Moreover, we are grateful to Paulo Yanez and three anonymmous referees for the detailed comments that have significantly improved this manuscript. 
TL gratefully acknowledges the financial support from the DFG Research Training Group 1845 ``Stochastic Analysis with Applications in Biology, Finance and Physics''. 
In addition, both authors gratefully acknowledge the financial support from the PROCOPE project ``Financial markets in transition: mathematical models and challenges''.\\}}

\abstract{We derive upper and lower bounds on the expectation of $f(\mathbf{S})$ under dependence uncertainty, i.e. when the marginal distributions of the random vector $\mathbf{S}=(S_1,\dots,S_d)$ are known but their dependence structure is partially unknown. 
We solve the problem by providing improved \FH bounds on the copula of $\mathbf{S}$ that account for additional information. 
In particular, we derive bounds when the values of the copula are given on a compact subset of $[0,1]^d$, the value of a functional of the copula is prescribed or different types of information are available on the lower dimensional marginals of the copula. 
We then show that, in contrast to the two-dimensional case, the bounds are quasi-copulas but fail to be copulas if $d>2$. 
Thus, in order to translate the improved \FH bounds into bounds on the expectation of $f(\mathbf{S})$, we develop an alternative representation of multivariate integrals with respect to copulas that admits also quasi-copulas as integrators. 
By means of this representation, we provide an integral characterization of orthant orders on the set of quasi-copulas which relates the improved \FH bounds to bounds on the expectation of $f(\mathbf{S})$. 
Finally, we apply these results to compute model-free bounds on the prices of multi-asset options that take partial information on the dependence structure into account, such as correlations or market prices of other traded derivatives. 
The numerical results show that the additional information leads to a significant improvement of the option price bounds compared to the situation where only the marginal distributions are known.}

\keyWords{Improved \FH bounds, quasi-copulas, stochastic dominance for quasi-copulas, model-free option pricing.}

\ArXiV{1602.08894}
\keyAMSClassification{62H05, 60E15, 91G20.}

\date{}
\frenchspacing \maketitle

\section{Introduction}

In recent years model uncertainty and uncertainty quantification have become ever more important topics in many areas of applied mathematics. 
Where traditionally the focus was on computing quantities of interest given a certain model, one today faces more frequently the challenge of estimating quantities in the absence of a fully specified model. 
In a probabilistic setting, one is interested in the expectation of $f(\mathbf{S})$, where $f\colon\mathbb{R}^d\to\mathbb{R}$ is a function and $\mathbf{S}=(S_1,\dots,S_d)$ is a random vector whose probability distribution is partially unknown. 
In this paper we consider the problem of finding upper and lower bounds on the expectation of $f(\mathbf{S})$ when the marginal distributions $F_i$ of $S_i$ are known while the dependence structure of $\mathbf{S}$ is partially unknown. 
This setting is referred to in the literature as \textit{dependence uncertainty}. 
The problem has an extensive history and several approaches to its solution have been developed. 
In the two-dimensional case, \citet{makarov} solved the problem for the quantile function of $f(x_1,x_2) = x_1+x_2$, while  \citet{rueschendorf2} considered more general functions $f$ fulfilling some monotonicity requirements. 
Both focused on the situation of complete dependence uncertainty, i.e. when no information on the dependence structure of $\mathbf{S}$ is available. 
Since then, solutions to this problem have evolved predominantly along the lines of optimal transportation, optimization theory and \FH bounds.

In this paper we take the latter approach to solving the problem in $d\ge2$ dimensions and for functions $f$ satisfying certain monotonicity properties. 
Assuming that the marginal distributions $F_i$ of $S_i$ are known and applying Sklar's Theorem, the problem can be reformulated as a minimization or maximization problem over the class of copulas that are compatible with the available information on $\mathbf{S}$. 
Using results from the theory of multivariate stochastic orders, bounds on the set of copulas can then be translated into bounds on the expectation of $f(\mathbf{S})$.

In the case of complete dependence uncertainty, that is, when only the marginals are known and no information on the joint behavior of the constituents of $\mathbf{S}$ is available, the bounds on the set of copulas are given by the well-known \FH bounds. 
They can however be improved in the presence of additional information on the copula. In case $d=2$, \citet{nelsen} derived improved \FH bounds if the copula of $\mathbf{S}$ is known at a single point. 
Similar improvements of the bivariate \FH bounds were provided by \citet{Rachev_Rueschendorf_1994} when the copula is known on an arbitrary set and by \citet{nelsen2} for the case in which a measure of association such as Kendall's $\tau$ or Spearman's $\rho$ is prescribed. 
\citet{tankov} recently generalized these results by improving the bivariate \FH bounds if the copula is known on a compact set or the value of a monotonic functional of the copula is prescribed. 
Since the bounds are in general not copulas but quasi-copulas, Tankov also provided sufficient conditions under which the improved bounds are copulas. 

In Sections 3 and 4 we establish improved \FH bounds on the set of $d$-dimensional copulas whose values are known on an arbitrary compact subset of 
$[0,1]^d$. 
Moreover, we provide analogous improvements when the value of a functional of the copula is prescribed or different types of information are available on the lower-dimensional margins of the copula. 
We further show that the improved bounds are quasi-copulas but fail to be copulas under fairly general assumptions. 
This constitutes a significant difference between the high-dimensional and the bivariate case, in which \citet{tankov} and \citet{bernard} showed that the improved bounds are copulas under quite relaxed conditions. 

Since our improved \FH bounds are merely quasi-copulas, results from stochastic order theory which translate bounds on the copula of $\mathbf{S}$ into bounds on the expectation of $f(\mathbf{S})$ do not apply. 
Even worse, the integrals with respect to quasi-copulas are not well-defined. 
Therefore, we derive in Section 5 an alternative representation of multivariate integrals with respect to copulas which admits also quasi-copulas as integrators, and establish integrability and continuity properties of this representation. 
Moreover, we provide an integral characterization of the lower and upper orthant order on the set of quasi-copulas, analogous to previous results on integral stochastic orders for copulas. 
These orders generalize the concept of first order stochastic dominance for multvariate distributions. 
Our results show that the representation of multivariate integrals is monotonic with respect to the upper or lower orthant order on the set of quasi-copulas for a large class of integrands. 
This enables us to compute bounds on the expectation of $f(\mathbf{S})$ that account for the available information on the marginal distributions and the copula of $\mathbf{S}$.

Finally, we apply our results in order to compute bounds on the prices of European, path-independent options in the presence of dependence uncertainty. 
These bounds are typically called \textit{model-free} or \textit{model-independent} in the literature, since no probabilistic model is assumed for the marginals or the dependence structure.
More specifically, we assume that $\mathbf{S}$ models the terminal value of financial assets whose risk-free marginal distributions can be inferred from market prices of traded vanilla options on its constituents. 
Moreover, we suppose that additional information on the dependence structure of $\mathbf{S}$ can be obtained from prices of traded derivatives on $\mathbf{S}$ or a subset of its components. 
This could be, for instance, information about the pairwise correlations of the components or prices of traded multi-asset options. 
Then, the improved \FH bounds and the integral characterization of orthant orders allow us to efficiently compute bounds on the set of arbitrage-free prices of $f(\mathbf{S})$ that are compatible with the available information on the distribution of $\mathbf{S}$. 
The payoff function $f$ should satisfy certain monotonicity conditions that hold for a plethora of options, such as digitals and options on the mininum or maximum of several assets, however basket options are excluded. 
In addition, the obtained bounds are not sharp in general. 
However, the numerical results show that the improved \FH bounds that take additional dependence information into account lead to a significant improvement of the option price bounds compared to the ones obtained from the `standard' \FH bounds.

\section{Notation and preliminary results}

In this section we introduce the notation and some basic results that will be used throughout this work. 
Let $d\geq2$ be an integer. 
In the sequel, $\mathbb{I}$ denotes the unit interval $[0,1]$, $\mathbf1$ denotes the vector with all entries equal to one, i.e. $\mathbf 1=(1,\dots,1)$, while boldface letters, e.g. $\mathbf{u}$, $\mathbf{v}$ or $\mathbf{x}$, denote vectors in $\mathbb{I}^d$, $\mathbb{R}^d$ or $\overline{\mathbb{R}}^d = [-\infty,\infty]^d$. 
Moreover, $\subseteq$ denotes the inclusion between sets and $\subset$ the proper inclusion, while we refer to functions as increasing when they are not decreasing. 

The finite difference operator $\Delta$ will be used frequently. 
It is defined for a function $f\colon\mathbb{R}^d\to\mathbb{R}$ and $a,b\in\mathbb{R}$ with $a\leq b$ via
$$\Delta_{a,b}^i\ f(x_1,\dots,x_d)  
  = f(x_1,\dots,x_{i-1},b,x_{i+1},\dots,x_d)
  - f(x_1,\dots,x_{i-1},a,x_{i+1},\dots,x_d).$$

\begin{definition}
A function $f\colon\mathbb{R}^d\to\mathbb{R}$ is called $d$-\textit{increasing} if for all rectangular subsets $H = (a_1,b_1]\times\cdots\times(a_d,b_d]\subset\mathbb{R}^d$ it holds that
\begin{align}\label{volume}
V_f(H) := \Delta^d_{a_d,b_d}\circ\cdots\circ\Delta^1_{a_1,b_1}\ f \geq 0.
\end{align}
Analogously, a function $f$ is called $d$-\textit{decreasing} if $-f$ is $d$-increasing.
Moreover, $V_f(H)$ is called the $f$-\textit{volume} of $H$.
\end{definition}

\begin{definition}
A function $Q\colon\mathbb{I}^d\to\mathbb{I}$ is a $d$-\textit{quasi-copula} if the following properties hold:
\begin{enumerate}[label={$(\mathbf{QC1})$},leftmargin=!,labelwidth=\widthof{
\bfseries XXXX}]
\item $Q$ satisfies, for all $i\in\{1,\dots,d\}$, the boundary conditions 
      \label{cond:QC1}
      $$Q(u_1,\dots,u_i = 0,\dots,u_d)=0 
      \quad\text{ and }\quad
      Q(1,\dots,1,u_i,1,\dots,1) = u_i.$$ 
\end{enumerate}
\begin{enumerate}[label={$(\mathbf{QC2})$},leftmargin=!,labelwidth=\widthof{
\bfseries XXXX}]
\item $Q$ is increasing in each argument. \label{cond:QC2}
\end{enumerate}
\begin{enumerate}[label={$(\mathbf{QC3})$},leftmargin=!,labelwidth=\widthof{
\bfseries XXXX}]
\item $Q$ is Lipschitz continuous, i.e. for all 
      $\mathbf{u},\mathbf{v}\in\mathbb{I}^d$\label{cond:QC3}
      $$|Q(u_1,\dots,u_d)-Q(v_1,\dots,v_d)|\leq\sum_{i=1}^d |u_i-v_i|.$$ 
\end{enumerate}
Moreover, $Q$ is a $d$-\textit{copula} if
\begin{enumerate}[label={$(\mathbf{QC4})$},leftmargin=!,labelwidth=\widthof{
\bfseries XXXX}]
\item $Q$ is $d$-increasing. \label{cond:QC4}
\end{enumerate}
\end{definition}

We denote the set of all $d$-quasi-copulas by $\mathcal{Q}^d$ and the set of all $d$-copulas by $\mathcal{C}^d$. 
Obviously $\mathcal{C}^d\subset\mathcal{Q}^d$. 
In the sequel, we will simply refer to a $d$-(quasi-)copula as (quasi-)copula if the dimension is clear from the context. 
Furthermore, we refer to elements in $\mathcal{Q}^d\setminus\mathcal{C}^d$ as proper quasi-copulas.

Let $C$ be a $d$-copula and consider $d$ univariate probability distribution functions $F_1,\dots,F_d$. 
Then $F(x_1,\dots,x_d):=C(F_1(x_1),\dots,F_d(x_d))$, for all $\mathbf{x}\in\mathbb{R}^d$, defines a $d$-dimensional distribution function with univariate margins $F_1,\dots,F_d$. 
The converse also holds by Sklar's Theorem, cf. \citet{sklar}. 
That is, for each $d$-dimensional distribution function $F$ with univariate marginals $F_1,\dots,F_d$, there exists a copula $C$ such that $F(x_1,\dots,x_d) = C(F_1(x_1),\dots,F_d(x_d))$ for all $\mathbf{x}\in\mathbb{R}^d$. 
In this case, the copula $C$ is unique if the marginals are continuous. 
A simple and elegant proof of Sklar's Theorem based on the distributional transform can be found in \citet{rueschendorf}. 
Sklar's Theorem establishes a fundamental link between copulas and multivariate distribution functions. 
Thus, given a random vector we will refer to its copula, i.e. the copula corresponding to the distribution function of this random vector.

Let $Q$ be a copula. 
We define its \textit{survival function} as follows:
$$\widehat{Q}(u_1,\dots,u_d) 
  := V_Q((u_1,1]\times\cdots\times(u_d,1]),\quad \mathbf{u} \in \mathbb{I}^d.$$
The survival function is illustrated for $d=3$ below:
\begin{align*}
\widehat{Q}(u_1,u_2,u_3) 
 &= 1 - Q(u_1,1,1,) - Q(1,u_2,1) - Q(1,1,u_3)\\
 &\quad + Q(u_1,u_2,1) + Q(u_1,1,u_3) + Q(1,u_2,u_3) - Q(u_1,u_2,u_3).
\end{align*}
A well-known result states that if $C$ is a copula then the function $\mathbf u \mapsto \widehat{C}(\mathbf1-\mathbf u)$, $\mathbf{u} \in \mathbb{I}^d$, is again a copula, namely the \textit{survival copula} of $C$; see e.g. \citet{georges}.
In contrast, if $Q$ is a quasi-copula then $\mathbf u \mapsto \widehat{Q}(\mathbf1-\mathbf u)$ is not a quasi-copula in general; see Example \ref{ex:counter} for a counterexample. 
We will refer to functions $\widehat{Q}\colon\mathbb{I}^d\to\mathbb{I}$  as \textit{quasi-survival functions} when $\bu\mapsto\widehat{Q}(\mathbf 1-\bu)$ is a quasi-copula. 
Let us point out that for a distribution function $F$ of a random vector $\mathbf{S}=(S_1,\dots,S_d)$ with marginals $F_1,\dots,F_d$ and a corresponding copula $C$ such that $F(x_1,\dots,x_d) = C(F_1(x_1),\dots,F_d(x_d))$ it holds that 
\begin{align}
\label{survivalCopulaProbability}
\mathbb{P}(S_1>x_1,\dots,S_d>x_d) = \widehat{C}(F_1(x_1),\dots,F_d(x_d)).
\end{align}

\begin{definition}
Let $Q_1,Q_2$ be $d$-quasi-copulas. 
$Q_2$ is larger than $Q_1$ in the \textit{lower orthant order}, denoted by $Q_1\preceq_{LO} Q_2$, if $Q_1(\mathbf{u})\leq Q_2(\mathbf{u})$ for all $\mathbf{u}\in\mathbb{I}^d$. 
Analogously, $Q_2$ is larger than $Q_1$ in the \textit{upper orthant order}, denoted by $Q_1\preceq_{UO} Q_2$ if $\widehat{Q}_1(\mathbf{u})\leq \widehat{Q}_2(\mathbf{u})$ for all $\mathbf{u}\in\mathbb{I}^d$. 
Moreover, the \textit{concordance order} is defined via $\preceq_{UO}$ and $\preceq_{LO}$, namely $Q_2$ is larger than $Q_1$ in concordance order if $Q_1\preceq_{UO} Q_2$ and $Q_1\preceq_{LO} Q_2$. 
\end{definition}

\begin{remark}
The lower and the upper orthant orders coincide when $d=2$. 
Hence, they also coincide with the concordance order.
\end{remark}

The well-known Fr\'{e}chet--Hoeffding theorem establishes the minimal and maximal bounds on the set of copulas or quasi-copulas in the lower orthant 
order. 
In particular, for each $Q\in\mathcal{C}^d$ or $Q\in\mathcal{Q}^d$, it holds that
$$W_d(\mathbf{u})
  := \max\Big\{0,\sum_{i=1}^d u_i - d + 1\Big\} 
  \leq   Q(\mathbf{u}) \leq \min\{u_1,\dots,u_d\} 
  =: M_d(\mathbf{u}),$$
for all $\mathbf{u}\in\mathbb{I}^d$, i.e. $W_d\preceq_{LO} Q\preceq_{LO} M_d$, where $W_d$ and $M_d$ are the lower and upper Fr\'{e}chet--Hoeffding bounds 
respectively. 
The upper bound is a copula for all $d\geq 2$, whereas the lower bound is a copula only if $d=2$ and a proper quasi-copula otherwise. 
A proof of this theorem can be found in \citet{genest}. 

A bound over a set of copulas, resp. quasi-copulas, is called \textit{sharp} if it belongs again to this set. 
Thus, the upper \FH bound is sharp for the set of copulas and quasi-copulas. 
Although the lower bound is not sharp for the set of copulas unless  $d=2$, it is (pointwise) \textit{best-possible} for all $d\in\mathbb{N}$ in the following 
sense:
$$W_d(\mathbf{u}) = \inf_{C\in\mathcal{C}^d} C(\mathbf{u}), 
  \quad \mathbf{u}\in\mathbb{I}^d;$$
cf. Theorem 6 in \citet{rueschendorf3}.

Since the properties of the Fr\'{e}chet--Hoeffding bounds carry over to the set of survival copulas in a straightforward way, one obtains similarly for any 
$C\in\mathcal{C}^d$ bounds with respect to the upper orthant order as follows:
$$W_d(1-u_1,\dots,1-u_d) \leq \widehat{C}(u_1,\dots,u_d) \leq 
  M_d(1-u_1,\dots,1-u_d), \quad\text{for all }\mathbf{u}\in\mathbb{I}^d.$$

\begin{example}\label{ex:counter}
Consider the lower \FH bound in dimension 3, i.e. $W_3$.
Then, $W_3$ is a quasi-copula, however $\bu \mapsto W_3(\mathbf1-\bu)$ is \textit{not} a quasi-copula again.
To this end, notice that quasi-copulas take values in $[0,1]$, while
$$
\widehat{W_3}\Big(\frac{1}{2},\frac{1}{2},\frac{1}{2}\Big) 
	= -\frac{1}{2}.
$$
\end{example}

\section{Improved Fr\'{e}chet--Hoeffding bounds under partial information on the dependence structure}
\label{sec:iFH}

In this section we develop bounds on $d$-copulas that improve the classical Fr\'{e}chet--Hoeffding bounds by assuming that partial information on the dependence structure is available. 
This information can be the knowledge either of the copula on a subset of $\mathbb{I}^d$, or of a measure of association, or of some lower-dimensional marginals of the copula. 
Analogous improvements can be obtained for the set of survival copulas in the presence of additional information and the respective results are presented in Appendix \ref{AppA}. 
The first result provides improved \FH bounds assuming that the copula is known on a subset of $\mathbb{I}^d$. 
The corresponding bounds for $d=2$ have been provided by \citet{Rachev_Rueschendorf_1994}, \citet{nelsen} and \citet{tankov}.

\begin{theorem}\label{PrescriptionOnSubset}
Let $\set\subset\mathbb{I}^d$ be a compact set and $Q^*$ be a $d$-quasi-copula. 
Consider the set 
\begin{align*}
\mathcal{Q}^{\set,Q^*} := \big\{ Q\in\mathcal{Q}^d\colon Q(\mathbf{x}) = 
  Q^*(\mathbf{x}) \text{ for all } \mathbf{x}\in\set\big\}.
\end{align*}
Then, for all $Q\in\mathcal{Q}^{\set,Q^*}$, it holds that 
\begin{align}\label{inequality}
\begin{split}
&\underline{Q}^{\set,Q^*}(\mathbf{u}) \leq 
  Q(\mathbf{u}) \leq \overline{Q}^{\set,Q^*}(\mathbf{u}) 
  \quad\text{for all }\mathbf{u}\in\mathbb{I}^d,\\
&\underline{Q}^{\set,Q^*}(\mathbf{u}) = 
  Q(\mathbf{u}) = \overline{Q}^{\set,Q^*}(\mathbf{u})
  \quad\text{for all }\mathbf{u}\in \set,
\end{split}
\end{align}
where the bounds $\underline{Q}^{\set,Q^*}$ and $\overline{Q}^{\set,Q^*}$ are provided by 
\begin{align}
\label{bounds}
\begin{split}
&\underline{Q}^{\set,Q^*}(\mathbf{u}) 
  = \max\Big(0, \sum_{i=1}^d u_i-d+1,\max_{\mathbf{x}\in \set} 
    \Big\{Q^*(\mathbf{x})-\sum_{i=1}^d (x_i-u_i)^+\Big\}\Big),\\
&\overline{Q}^{\set,Q^*}(\mathbf{u}) 
  = \min\Big(u_1,\dots,u_d,\min_{\mathbf{x}\in \set} 
    \Big\{Q^*(\mathbf{x})+\sum_{i=1}^d (u_i-x_i)^+\Big\}\Big).
\end{split}
\end{align}
Furthermore, the bounds $\underline{Q}^{\set,Q^*},\overline{Q}^{\set,Q^*}$ are $d$-quasi-copulas, hence they are sharp.
\end{theorem}

\begin{proof}
We start by considering a prescription at a single point, i.e. we let $\set=\{\mathbf{x}\}$ for $\mathbf{x}\in\mathbb{I}^d$, and show that $\underline{Q}^{\{\mathbf x\},Q^*}$ and $\overline{Q}^{\{\mathbf x\},Q^*}$, provided by \eqref{bounds} for $\set=\{\mathbf x\}$, satisfy \eqref{inequality} for all $Q\in\mathcal Q^{\{\bx\},Q^*}$.
In this case, analogous results were provided by \citet{Rodriguez_Ubeda_2004}. 
We present below a simpler, direct proof. 
Let $Q\in\mathcal{Q}^{\{\mathbf{x}\},Q^*}$ be arbitrary and $(u_1,\dots,u_d)$, $(u_1,\dots,u_{i-1},x_i,u_{i+1},\dots,u_d)\in\mathbb{I}^d$, then it follows 
from the Lipschitz property of $Q$ and the fact that $Q$ is increasing in each 
coordinate that
$$-(u_i-x_i)^+\leq 
Q(u_1,\dots,u_{i-1},x_i,u_{i+1},\dots,u_d)-Q(u_1,\dots,u_d)\leq (x_i-u_i)^+.$$
Using the telescoping sum 
\begin{align*}
Q(x_1,\dots,x_d)&-Q(u_1,\dots,u_d) = 
Q(x_1,\dots,x_d)-Q(u_1,x_2,\dots,x_d)+Q(u_1,x_2,\dots,x_d)\\
&-Q(u_1,u_2,x_3,\dots,x_d)+\cdots+Q(u_1,\dots,u_{d-1},x_d)-Q(u_1,\dots,u_d)
\end{align*}
we arrive at 
$$-\sum_{i=1}^d(u_i-x_i)^+ \leq Q(x_1,\dots,x_d)-Q(u_1,\dots,u_d) 
\leq \sum_{i=1}^d(x_i-u_i)^+$$
which is equivalent to
$$Q(x_1,\dots,x_d) - \sum_{i=1}^d(x_i-u_i)^+
  \leq Q(u_1,\dots,u_d)  
  \leq Q(x_1,\dots,x_d) +\sum_{i=1}^d(u_i-x_i)^+.$$
The prescription yields further that $Q(x_1,\dots,x_d)=Q^*(x_1,\dots,x_d)$, from which follows that
$$Q^*(x_1,\dots,x_d) - \sum_{i=1}^d(x_i-u_i)^+ 
	\leq Q(u_1,\dots,u_d)  
	\leq Q^*(x_1,\dots,x_d) + \sum_{i=1}^d(u_i-x_i)^+,$$
while incorporating the Fr\'{e}chet--Hoeffding bounds yields
\begin{multline}
\max\Big\{0,\sum_{i=1}^d u_i-d+1, Q^*(x_1,\dots,x_d) 
	- \sum_{i=1}^d(x_i-u_i)^+\Big\} 
	\leq Q(u_1,\dots,u_d) \\
\leq \min\Big\{u_1,\dots,u_d,Q^*(x_1,\dots,x_d)+\sum_{i=1}^d(u_i-x_i)^+\Big\},
\end{multline}
showing that the inequalities in \eqref{inequality} are valid for $\set=\{\mathbf{x}\}$. 
Moreover, since $W_d(\mathbf{x}) \leq Q^*(\mathbf{x}) \leq M_d(\mathbf{x})$ it holds that
\begin{align*}
&\underline{Q}^{\{\bx\},Q^*}(\mathbf{x}) 
	= \max\Big\{0,\sum_{i=1}^d x_i-d+1,Q^*(x_1,\dots,x_d)\Big\} 
	= Q^*(\mathbf{x}),\\
&\overline{Q}^{\{\bx\},Q^*}(\mathbf{x}) 
	= \min\Big\{x_1,\dots,x_d,Q^*(x_1,\dots,x_d)\Big\} 
	= Q^*(\mathbf{x}),
\end{align*}
showing that the equalities in \eqref{inequality} are valid for $\set= \{\mathbf{x}\}$.

Next, let $\set$ be a compact set which is not a singleton and consider a $Q\in\mathcal{Q}^{\set,Q^*}$. 
We know from the arguments above that 
$Q(\mathbf{u})\geq\underline{Q}^{\{\mathbf{x}\},Q^*}(\mathbf{u})$ for all 
$\mathbf{x}\in \set$, therefore 
\begin{align}\label{maxinsideout}
Q(\mathbf{u}) 
  \geq \max_{\mathbf{x}\in \set} 
  \Big\{\underline{Q}^{\{\mathbf{x}\},Q^*}(\mathbf{u})\Big\} 
  = \underline{Q}^{\set,Q^*}(\mathbf{u}).
\end{align}
Analogously we get for the upper bound that 
\begin{align}\label{mininsideout}
Q(\mathbf{u})
  \leq \min_{\mathbf{x}\in \set} 
  \Big\{\overline{Q}^{\{\mathbf{x}\},Q^*}(\mathbf{u})\Big\}
  = \overline{Q}^{\set,Q^*}(\mathbf{u}),
\end{align}
hence the inequalities in \eqref{inequality} are valid. 
Moreover, if $\mathbf{u}\in \set$ then $Q(\mathbf{u}) = Q^*(\mathbf{u})$ for all $Q\in\mathcal{Q}^{\set,Q^*}$ and using the Lipschitz property of quasi-copulas we obtain
$$\max_{\mathbf{x}\in \set} \Big\{Q^*(\mathbf{x})-\sum_{i=1}^d(x_i-u_i)^+\Big\}
  = Q^*(\mathbf{u}) \,\,\text{ and }\,\,
  \min_{\mathbf{x}\in \set} \Big\{Q^*(\mathbf{x})+\sum_{i=1}^d(u_i-x_i)^+\Big\}
  = Q^*(\mathbf{u}).
$$
Hence, using again that $Q^*$ satisfies the Fr\'echet--Hoeffding bounds we arrive at 
$$\underline{Q}^{\set,Q^*}(\mathbf{u}) 
 	= Q(\mathbf{u})
	= \overline{Q}^{\set,Q^*}(\mathbf{u}).
$$	

Finally, it remains to show that both bounds are $d$-quasi-copulas.
\begin{itemize}
\item[$\bullet$] 
	In order to show that \ref{cond:QC1} holds, first consider the case $\set=\{x\}$. Let $(u_1,\dots,u_d)\in\mathbb{I}^d$ with $u_i=0$ for one $i\in\{1,\dots,d\}$. 
	Then $\overline{Q}^{\set,Q^*}(\mathbf{u})$ is obviously zero, and $\underline{Q}^{\set,Q^*}(\mathbf{u})=\max\big\{0,Q^*(\mathbf{x})-x_i-\sum_{j\neq i}(x_j-u_j)^+\big\}=0$ because $Q^*(\mathbf{x})\leq \min\{x_1,\dots,x_d\}$, i.e. $Q^*(\mathbf{x})-x_i-\sum_{j\neq i}(x_j-u_j)^+\leq 0$ for all $\mathbf{x}\in \set$. 
	Moreover for 
  $(u_1,\dots,u_d)\in\mathbb{I}^d$ with $u_i=1$ for all 
  $i\in\{1,\dots,d\}\setminus\{j\}$, the upper bound equals 
  $\overline{Q}^{\set,Q^*}(\mathbf{u})=\min\big\{u_j,Q^*(\mathbf{x})+\sum_{i=1}
  ^d(u_i-x_i)^+\big\}$ and since 
\begin{align*}
Q^*(\mathbf{x})+\sum_{i=1}^d (u_i-x_i)^+ 
&= Q^*(\mathbf{x})+\sum_{i\in\{1,\dots,d\}\setminus \{j\}} (1-x_i)+(u_j-x_j)^+\\
&= Q^*(\mathbf{x})+d-1-\sum_{i\in\{1,\dots,d\}\setminus \{j\}} x_i+(u_j-x_j)^+\\
&\geq W_d(\mathbf{x})+d-1-\sum_{i\in\{1,\dots,d\}\setminus \{j\}} 
  x_i+(u_j-x_j)^+\\
&\ge x_j+(u_j-x_j)^+\geq u_j,
\end{align*}	
  it follows that $Q^*(\mathbf{u})=u_j$, hence $\overline{Q}^{\set,Q^*}(\mathbf{u})=u_j$. 
  Similarly, the lower bound amounts to $\underline{Q}^{\set,Q^*}(\mathbf{u})=\max\big\{0,u_j,Q^*(\mathbf{x})-(x_j-u_j)^+\big\}$ which equals $u_j$ because $Q^*(\mathbf{x})-(x_j-u_j)^+\leq M_d(\mathbf{x})-(x_j-u_j)^+\leq u_j$. 
  The boundary conditions hold analogously for $\set$ containing more than one element due to the continuity of the maximum and minimum functions and relationships \eqref{maxinsideout} and \eqref{mininsideout}. 
\item[$\bullet$] 
	Both bounds are obviously increasing in each variable, thus \ref{cond:QC2} holds.
\item[$\bullet$] 
	Finally, taking the pointwise minimum and maximum of Lipschitz functions preserves the Lipschitz property, thus both bounds satisfy \ref{cond:QC3}.
\end{itemize}
\end{proof}

\begin{remark}
The bounds in Theorem \ref{PrescriptionOnSubset} hold analogously for prescriptions on copulas, i.e. for all $C$ in $\mathcal C^{\set,Q^*}=\{C\ \in\mathcal{C}^d\colon C(\mathbf{x}) = Q^*(\mathbf{x}) \text{ for all } \mathbf{x}\in \set\}$ where $Q^*$ and $\set$ are as above, it holds that $\underline{Q}^{\set,Q^*}(\mathbf{u}) \leq C(\mathbf{u}) \leq \overline{Q}^{\set,Q^*}(\mathbf{u})$ for all $\bu\in\mathbb{I}^d$.
Let us point out that the set $\mathcal{C}^{\set,Q^*}$ may possibly be empty, depending on the prescription. 
We will not investigate the requirements on the prescription for $\mathcal{C}^{\set,Q^*}$ to be non-empty. 
A detailed discussion of this issue in the two-dimensional case can be found in \citet{mardani-fard}. 
\end{remark}

Next, we derive improved bounds on $d$-quasi-copulas when values of real-valued functionals of the quasi-copulas are prescribed. 
Examples of such functionals are the multivariate generalizations of Spearman's rho and Kendall's tau given in \citet{taylor}. 
Moreover, in the context of multi-asset option pricing, examples of such functionals are prices of spread or digital options. 
Analogous results for $d=2$ are provided by \citet{nelsen} and \citet{tankov}.

\begin{theorem}
\label{prescribedFunctional}
Let $\rho\colon\mathcal{Q}^d\to\mathbb{R}$ be increasing with respect to the lower orthant order on $\mathcal{Q}^d$ and continuous with respect to the 
pointwise convergence of quasi-copulas. 
Define 
$${\mathcal{Q^{\rho,\theta}}:=\{Q\in\mathcal{Q}^d\colon \rho(Q)=\theta\}}$$ 
for $\theta\in(\rho(W_d),\rho(M_d))$.
Then the following bounds hold
\begin{align*}
\underline{Q}^{\rho,\theta}(\mathbf{u}) 
:= \min\big\{Q(\mathbf{u})\colon Q\in\mathcal{Q}^{\rho,\theta}\big\}
=\begin{cases}
  \rho_+^{-1}(\mathbf{u},\theta),
  & \theta\in\big[\rho\big(\overline{Q}^{\{\mathbf{u}\},W_d(\mathbf{u})}\big),
    \rho(M_d)\big],\\
  W_d(\mathbf{u}),
  & \text{else},
\end{cases}
\end{align*}
and
\begin{align*}
\overline{Q}^{\rho,\theta}(\mathbf{u}) 
:= \max\big\{Q(\mathbf{u})\colon Q\in\mathcal{Q}^{\rho,\theta}\big\}
=\begin{cases}
  \rho_-^{-1}(\mathbf{u},\theta),
  &\theta\in\big[\rho(W_d),\rho\big(\underline{Q}^{\{\mathbf{u}\}, 
    M_d(\mathbf{u})}\big)\big],\\
  M_d(\mathbf{u}),
  & \text{else},
\end{cases}
\end{align*}
and these are again quasi-copulas. 
Here
\begin{align*}
\rho_-^{-1}(\mathbf{u},\theta) 
  = \max\big\{r\colon \rho\big(\underline{Q}^{\{\mathbf{u}\},r}\big) 
    = \theta\big\}
\qquad\text{and}\qquad 
\rho_+^{-1}(\mathbf{u},\theta) 
  = \min\big\{r\colon \rho\big(\overline{Q}^{\{\mathbf{u}\},r}\big)
    = \theta\big\},
\end{align*}
while the quasi-copulas $\underline{Q}^{\{\mathbf{u}\},r}$ and $\overline{Q}^{\{\mathbf{u}\},r}$ are given in Theorem \ref{PrescriptionOnSubset} for $r\in\mathbb{I}$.

\begin{proof}
We will show that the upper bound is valid, while the proof for the lower bound follows analogously. 
First, note that due to the continuity of $\rho$ w.r.t. the  pointwise convergence of quasi-copulas and the compactness of $\mathcal{Q}^d$, we get that the set $\big\{Q(\mathbf{u})\colon Q\in\mathcal{Q}^{\rho,\theta}\big\}$ is compact, hence
$$
\sup\big\{Q(\mathbf{u})\colon Q\in\mathcal{Q}^{\rho,\theta}\big\} = \max\big\{Q(\mathbf{u})\colon Q\in\mathcal{Q}^{\rho,\theta}\big\}.
$$
Next, let $\theta\in\big[\rho(W_d),\rho\big(\underline{Q}^{\{\mathbf{u}\},M_d(\mathbf{u})}\big)\big]$, then $\overline{Q}^{\rho,\theta}(\mathbf{u})\leq\rho_-^{-1}(\mathbf{u},\theta)$ due to the construction of $\rho_-^{-1}(\mathbf{u},\theta)$. 
Moreover it holds that $\rho\big(\underline{Q}^{\{\mathbf{u}\},\rho_-^{-1}(\mathbf{u},\theta)}\big) = \theta$ since $r\mapsto\rho\big(\underline{Q}^{\{\mathbf{u}\},r}\big)$ is increasing and continuous, therefore $\overline{Q}^{\rho,\theta}(\mathbf{u})\geq\rho_-^{-1}(\mathbf{u},\theta)$. 
Hence, we can conclude that $\overline{Q}^{\rho,\theta}(\mathbf{u})=\rho_-^{-1}(\mathbf{u},\theta)$ whenever $\theta\in[\rho(W_d),\rho\big(\underline{Q}^{\{\mathbf{u}\},M_d(\mathbf{u})}\big)]$.

Now, let $\theta>\rho\Big(\underline{Q}^{\{\mathbf{u}\},M_d(\mathbf{u})}\Big)$, then $\theta\in\Big(\rho\Big(\underline{Q}^{\{\mathbf{u}\},M_d(\mathbf{u})}\Big),\rho(M_d)\Big]$. 
Consider $Q^\alpha = \alpha M_d + (1-\alpha)\underline{Q}^{\{\mathbf{u}\},M_d(\mathbf{u})}$, for $\alpha\in[0,1]$, then $\rho(Q^0) <\theta$ and $\rho(Q^1)\geq\theta$. 
Since $\alpha\mapsto \rho(Q^\alpha)$ is continuous there exists an $\alpha$ with $\rho(Q^\alpha)= \theta$. 
Since $Q^\alpha(\mathbf{u}) = M_d(\mathbf{u})$ for all $\alpha\in[0,1]$ it follows that $M_d(\mathbf{u})\leq \max\big\{Q(\mathbf{u})\colon Q\in\mathcal{Q}^{\rho,\theta}\big\}$, while the reverse inequality holds due to the upper \FH bound.

Finally, using Theorem 2.1 in \citet{Rodriguez_Ubeda_2004} we get immediately that the bounds are again quasi-copulas.
\end{proof}
\end{theorem}

\begin{remark}
The bounds in Theorem \ref{prescribedFunctional} hold analogously for copulas, i.e. for $\rho$ and $\theta$ as in Theorem \ref{prescribedFunctional} we have 
$\underline{Q}^{\rho,\theta}\preceq_{LO} C \preceq_{LO} \overline{Q}^{\rho,\theta}$ for all $C\in \{C\in\mathcal{C}^d\colon \rho(C) = \theta\}$.
\end{remark}

\begin{remark}
The bounds $\underline{Q}^{\rho,\theta}$ and $\overline{Q}^{\rho,\theta}$ do not belong to the set $\mathcal{Q}^{\rho,\theta}$ in general. 
A counterexample in dimension 2 is provided by combining \citet[Theorem 2]{tankov} with \citet[Corollary 3(h)]{nelsen2}.
Indeed, from the first reference we get that 
\begin{align*}
\underline{Q}^{\rho,\theta}(\mathbf{u}) = \min\big\{Q(\mathbf{u})\colon Q\in\mathcal{Q}^{\rho,\theta}\cap\mathcal{C}^2\big\}
\quad \text{and} \quad 
\overline{Q}^{\rho,\theta}(\bu) = \max\big\{Q(\mathbf{u})\colon Q\in\mathcal{Q}^{\rho,\theta}\cap\mathcal{C}^2\big\},
\end{align*}
while the second one shows that neither of these bounds belongs to $\mathcal{Q}^{\rho,\theta}$ when $\rho\in(-1,1)$, where $\rho$ stands for Kendall's tau in this case.
\end{remark}


In the next theorem we construct improved \FH bounds assuming that information only on some lower-dimensional marginals of a quasi-copula is available. 
This result corresponds to the situation where one is interested in a high-dimensional random vector, however information on the dependence structure is only available for lower-dimensional vectors thereof. 
As an example, in mathematical finance one is interested in options on several assets, however information on the dependence structure---stemming e.g. from other liquid option prices---is typically available only on pairs of those assets. 

Let us introduce a convenient subscript notation for the lower-dimensional marginals of a quasi-copula. 
Consider a subset $I=\{i_1,\dots,i_n\}\subset\{1,\dots,d\}$ and define the \textit{projection} of a vector $\bu\in\R^d$ to the lower-dimensional space $\R^n$ via 
$\bu_I:=(u_{i_1},\dots,u_{i_n})\in\R^n$. 
Moreover, define the \textit{lift} of the vector $\bu_I\in\R^n$ to the higher-dimensional space $\R^d$ by $\bu_I'=:\bv\in\R^d$ where $v_i=u_i$ if $i\in I$ and $v_i=1$ if $i\notin I$. 
Then, we can define the $I$-\textit{margin} of the $d$-quasi-copula $Q$ via $Q_I: \mathbb{I}^n \to \mathbb{I}$ with $\bu_I \mapsto Q(\bu_I').$ 

\begin{remark}\label{rem:pro-lift}
Let $\bu\in\mathbb{I}^d$ and $I\subset\{1,\dots,d\}$. 
Then, by first projecting $\bu$ and then lifting it back, we get that $\bu\le\bu_I'$ (where $\le$ denotes the component-wise order).  
Hence, by \ref{cond:QC2} we get that $Q(\bu) \le Q_I(\bu_I) = Q(\bu_I')$.
\end{remark}

\begin{theorem}
\label{regionalPrescription}
Let $I_1,\dots,I_k$ be subsets of $\{1,\dots,d\}$ with $|I_j|\geq 2$ for $j\in\{1,\dots,k\}$ and $|I_i\cap I_j| \leq 1$ for $i,j\in\{1,\dots,k\}$, $i\neq j$. 
Let $\underline{Q}_j,\overline{Q}_j$ be $|I_j|$-quasi-copulas with $\underline{Q}_j\preceq_{LO}\overline{Q}_j$ for $j=1,\dots,k$, and consider the set
$$\mathcal{Q}^{I} = \left\{ Q\in\mathcal{Q}^d\colon 
  \underline{Q}_j\preceq_{LO} Q_{I_j}\preceq_{LO} \overline{Q}_j,\ j=1,\dots,k 
  \right\},$$
where $Q_{I_j}$ are the $I_j$-margins of $Q$. 
Then $\mathcal{Q}^{I}$ is non-empty and the following bounds hold
\begin{align*}
&\underline{Q}^I(\mathbf{u}) 
  := \min\big\{Q(\mathbf{u})\colon Q\in\mathcal{Q}^{I}\big\}\\
& \qquad\qquad =\max\bigg( \max_{j\in\{1,\dots,k\}} 
  \Big\{\underline{Q}_j(\mathbf{u}_{I_j})+\sum_{l\in\{1,\dots,d\}\setminus 
  I_j} (u_l-1)\Big\},W_d(\mathbf{u})\bigg),\\
&\overline{Q}^I(\mathbf{u}) 
  := \max\big\{Q(\mathbf{u})\colon Q\in\mathcal{Q}^{I} \big\} 
  =\min\Big(\min_{j\in\{1,\dots,k\}}\big\{ 
  \overline{Q}_j(\mathbf{u}_{I_j})\big\}, M_{d} (\mathbf{u})\Big).
\end{align*} 
Moreover $\underline{Q}^I, \overline{Q}^I\in\mathcal{Q}^{I}$, hence the bounds are sharp.

\begin{proof}
Let $Q\in\mathcal{Q}^{I}$ and $\mathbf{u}\in\mathbb{I}^d$. 
We first show that the upper bound $\overline{Q}^I$ is valid. 
It follows directly from Remark \ref{rem:pro-lift} that
\begin{align*}
Q(\mathbf{u})\leq Q(\mathbf{u}'_{I_j}) = Q_{I_j}(\mathbf{u}_{I_j})
  \leq \overline{Q}_j(\mathbf{u}_{I_j}), \quad \text{for all } j=1,\dots,k,
\end{align*}
hence $Q(\mathbf{u}) \leq \min_{j\in\{1,\dots,k\}} \big\{\overline{Q}_j(\mathbf{u}_{I_j})\big\}$. 
Incorporating the upper \FH bound yields $\overline{Q}^I$. 
Moreover, \ref{cond:QC1} and \ref{cond:QC2} follow immediately since $\overline{Q}_j$ are quasi-copulas for $j=1,\dots,k$, while $\overline{Q}^I$ is a composition of Lipschitz functions and hence Lipschitz itself, i.e. \ref{cond:QC3} also holds. 
Thus $\overline{Q}^I$ is indeed a quasi-copula.

As for the lower bound, using once more the projection and lift operations and the Lipschitz property of quasi-copulas we have
\begin{align*}
Q(\mathbf{u}) 
  &\geq Q(\mathbf{u}'_{I_j})+\sum_{l\in\{1,\dots,d\}\setminus I_j} (u_l-1) 
  = Q_{I_j}(\mathbf{u}_{I_j})+\sum_{l\in\{1,\dots,d\}\setminus I_j}(u_l-1)\\
  &\geq \underline{Q}_j(\mathbf{u}_{I_j})+\sum_{l\in\{1,\dots,d\}\setminus I_j} 
    (u_l-1), \quad \text{for all } j=1,\dots,k.
\end{align*}
Therefore,
\begin{align}
Q(\mathbf{u}) 
  \geq \max_{j\in\{1,\dots,k\}} \Big\{\underline{Q}_j(\mathbf{u}_{I_j}) 
  + \sum_{l\in\{1,\dots,d\}\setminus I_j} (u_l-1)\Big\},
\end{align}
and including the lower \FH bound yields $\underline{Q}^I$. 
In order to verify that $\underline{Q}^I$ is a quasi-copula, first consider $\mathbf{u}\in\mathbb{I}^d$ with $u_i=0$ for at least one $i\in\{1,\dots,d\}$. 
Then $W_d(\mathbf{u})=0$,
$$\underline{Q}_j(\mathbf{u}_{I_j}) + \sum_{l\in\{1,\dots,d\}\setminus I_j} 
  (u_l-1) \leq \underline{Q}_j(\mathbf{u}_{I_j})-1 \leq 0 
  \quad\text{if }i\in\{1,\dots,d\}\setminus I_j,$$
and 
$$\underline{Q}_j(\mathbf{u}_{I_j}) + \sum_{l\in\{1,\dots,d\}\setminus I_j} 
  (u_l-1) = \sum_{l\in\{1,\dots,d\}\setminus I_j} (u_l-1) \leq 0
  \quad\text{if } i\in I_j,$$ 
for all $j=1,\dots,k$. Hence $\underline{Q}^I(\mathbf{u}) = 0$. 
In addition, for $\mathbf{u}\in\mathbb{I}^d$ with $\mathbf{u}=\mathbf{u}'_{\{i\}}$, it follows that $W_d(\mathbf{u}) = u_i$ and 
$$\underline{Q}_j(\mathbf{u}_{I_j}) + \sum_{l\in\{1,\dots,d\}\setminus I_j} 
  (u_l-1) = 1+(u_i-1) = u_i \quad\text{if }i\in\{1,\dots,d\}\setminus I_j,$$
while clearly $\underline{Q}_j(\mathbf{u}_{I_j})=u_i$ if $i\in I_j$, for all $j=1,\dots,d$. 
Hence $\underline{Q}^I(\mathbf{u}) = u_i$, showing that $\underline{Q}^I$ fulfills \ref{cond:QC1}. \ref{cond:QC2} is immediate, while noting that $\underline{Q}^I$ is a composition of Lipschitz functions and hence Lipschitz itself shows that the lower bound is also a $d$-quasi-copula.

Finally, knowing that $\underline{Q}^I,\overline{Q}^I$ are quasi-copulas it remains to show that both bounds are in $\mathcal{Q}^{I}$, i.e. we need to show 
that $\underline{Q}_j\preceq \big(\underline{Q}^I\big)_{I_j}, \big(\overline{Q}^I\big)_{I_j}\preceq \overline{Q}_j$ for all $j=1,\dots,k$. 
For the upper bound it holds by definition that $\big(\overline{Q}^I\big)_{I_j} \preceq \overline{Q}_j$ for $j=1,\dots,k$. 
Moreover since $|I_i\cap I_j| \leq 1$ it follows that $ \big(\overline{Q}^I\big)_{I_j} = \overline{Q}_j$, hence $\underline{Q}_j \preceq \big(\overline{Q}^I\big)_{I_j} \preceq \overline{Q}_j$ for $j=1,\dots,k$ and $ \overline{Q}^I\in\mathcal{Q}^{I}$. 
By the same argument it holds for the lower bound that $\big(\underline{Q}^I\big)_{I_j} = \underline{Q}_j$ for $j=1,\dots,d$, thus $\underline{Q}_j\preceq 
\big(\underline{Q}^I\big)_{I_j}\preceq \overline{Q}_j$ for $j=1,\dots,k$, showing that $\underline{Q}_j\preceq \big(\underline{Q}^I\big)_{I_j}, \big(\overline{Q}^I\big)_{I_j}\preceq \overline{Q}_j$ holds indeed.
\end{proof}
\end{theorem}

\begin{remark}
\label{regionalPrescriptionRemark}
The bounds in Theorem \ref{regionalPrescription} hold analogously for copulas. 
That is, for subsets $I_1,\dots.,I_k$ and quasi-copulas $\underline{Q}_j,\overline{Q}_j$ as in Theorem \ref{regionalPrescription} and defining
$$\mathcal{C}^{I} := \left\{C\in\mathcal{C}^d\colon 
  \underline{Q}_j\preceq_{LO} C_{I_j}\preceq_{LO} \overline{Q}_j,\ j=1,\dots,k 
  \right\}$$
it follows that $\underline{Q}^I\preceq_{LO} C \preceq_{LO}\overline{Q}^I$ for all $C\in \mathcal{C}^{I}$. 
\end{remark}

\section{Are the improved \FH bounds copulas?}
\label{boundsQuasiCopulas}

An interesting question arising now is under what conditions the improved \FH bounds are copulas and not merely quasi-copulas. 
This would allow us, for example, to translate those bounds on the copulas to bounds on the expectations with respect to the underlying random variables. 
\citet{tankov} showed that if $d=2$, then $\underline{Q}^{\set,Q^*}$ and $\overline{Q}^{\set,Q^*}$ are copulas under certain constraints on the set $\set$. 
In particular, if $\set$ is increasing (also called comonotone), that is if $(u_1,u_2),(v_1,v_2)\in \set$ then $(u_1-v_1)(u_2-v_2)\ge0$ holds, then the lower bound $\underline{Q}^{\set,Q^*}$ is a copula. 
Conversely, if $\set$ is decreasing (also called countermonotone), that is if $(u_1,u_2),(v_1,v_2)\in \set$ then $(u_1-v_1)(u_2-v_2)\le0$ holds, then the upper bound $\overline{Q}^{\set,Q^*}$ is a copula. 
\citet{bernard} relaxed these constraints and provided minimal conditions on $\set$ such that the bounds are copulas. 
The situation however is more complicated for $d>2$. On the one hand, the notion of a decreasing set is not clear. 
On the other hand, the following counterexample shows that the condition of $\set$ being an increasing set is not sufficient for $\underline{Q}^{\set,Q^*}$ to be a copula. 

\begin{example}
Let $\set=\big\{(u,u,u)\colon u\in\big[0,\frac12\big]\cup\big[\frac35,1\big]\big\}\subset\mathbb{I}^3$ and $Q^*$ be the independence copula, i.e. $Q^*(u_1,u_2,u_3) = u_1u_2u_3$ for $(u_1,u_2,u_3)\in\mathbb{I}^3$. 
Then $\set$ is clearly an increasing set, however $\underline{Q}^{\set,Q^*}$ is not a copula. 
To this end, it suffices to show that the $\underline{Q}^{\set,Q^*}$-volume of some subset of $\mathbb{I}^3$ is negative. 
Indeed, for $\big[\frac{56}{100},\frac35\big]^3\subset\mathbb{I}^3$ after some straightforward calculations we get that
\begin{multline*}
V_{\underline{Q}^{\set,Q^*}}\left(\left[\frac{56}{100},\frac35\right]^3\right) 
= \left(\frac35\right)^3 
  - 3\left[\left(\frac35\right)^3 - \left(\frac35-\frac{56}{100}\right)\right]\\
 + 3\left[\left(\frac35\right)^3 -2\left(\frac35-\frac{56}{100}\right)\right]
  - \left(\frac12\right)^3
= -0.029 < 0.
\end{multline*}
\end{example}

In the trivial case where $\set=\mathbb{I}^d$ and $Q^*$ is a $d$-copula, then both bounds from Theorem \ref{PrescriptionOnSubset} are copulas for $d>2$ since they equate to $Q^*$. 
Moreover, the upper bound is a copula for $d>2$ if it coincides with the upper \FH bound. 
The next result shows that essentially only in these trivial situations are the bounds copulas for $d>2$. 
Out of instructive reasons we first discuss the case $d=3$, and defer the general result for $d>3$ to Appendix \ref{AppB}. 

\begin{theorem}
\label{lowerBoundIsQuasiCopula}
Consider the compact subset $\set$ of $\I^3$
\begin{align}
\label{minimalS}
\set &={\Big([0,1]\setminus(s_1,s_1+\varepsilon_1)\Big)} 
  \times {\Big([0,1]\setminus(s_2,s_2+\varepsilon_2)\Big)}
  \times {\Big([0,1]\setminus(s_3,s_3+\varepsilon_3)\Big)},
\end{align}
for $\varepsilon_i>0$, $i=1,2,3$ and let $C^*$ be a $3$-copula (or a 
$3$-quasi-copula) such that 
\begin{align}
& \sum_{i=1}^3 \varepsilon_i > 
  C^*(\mathbf{s}+\pmb{\varepsilon})-C^*(\mathbf{s})>0, \label{req1.1}\\
& C^*(\mathbf{s}) \geq  W_3(\mathbf{s}+\pmb{\varepsilon}) \label{req1.2},
\end{align}
where 
$\mathbf{s}=(s_1,s_2,s_3),\pmb{\varepsilon}=(\varepsilon_1,\varepsilon_2,
\varepsilon_3)$. Then $\underline{Q}^{\set,C^*}$ is a proper quasi-copula.
\end{theorem}

\begin{proof}
Assume that $C^*$ is a $d$-copula and choose 
$\mathbf{u}=(u_1,u_2,u_3)\in(\mathbf{s},\mathbf{s}+\pmb{\varepsilon})$ such that
\begin{align}
&C^*(\mathbf{s}+\pmb{\varepsilon})-C^*(\mathbf{s}) < \sum_{i=1}^3 
(s_i+\varepsilon_i -u_i)\quad\text{and}\label{req1.3}\\
&C^*(\mathbf{s}+\pmb{\varepsilon})-C^*(\mathbf{s}) > \sum_{i\in J} 
(s_i+\varepsilon_i -u_i)\quad \text{for } J = (1,2),(2,3),(1,3); \label{req1.4}
\end{align}
such a $\mathbf{u}$ exists due to \eqref{req1.1}. In order to show that 
$\underline{Q}^{\set,C^*}$ is not 3-increasing, and thus not a (proper) copula, 
it suffices to prove that 
$V_{\underline{Q}^{\set,C^*}}([\mathbf{u},\mathbf{s}+\pmb{\varepsilon}])<0$. By 
the definition of $V_{\underline{Q}^{\set,C^*}}$ we have
\begin{align*}
V_{\underline{Q}^{\set,C^*}}([\mathbf{u},\mathbf{s}+\pmb{\varepsilon}]) 
  &= \underline{Q}^{\set,C^*}(\mathbf{s}+\pmb{\varepsilon}) 
    -\underline{Q}^{\set,C^*}(u_1,s_2+\varepsilon_2,s_3+\varepsilon_3)\\
  &\quad -\underline{Q}^{\set,C^*}(s_1+\varepsilon_1,u_2,s_3+\varepsilon_3)
    -\underline{Q}^{\set,C^*}(s_1+\varepsilon_1,s_2+\varepsilon_2,u_3)\\
  &\quad +\underline{Q}^{\set,C^*}(u_1,u_2,s_3+\varepsilon_3)
    +\underline{Q}^{\set,C^*}(u_1,s_2+\varepsilon_2,u_3)\\
  &\quad +\underline{Q}^{\set,C^*}(s_1+\varepsilon_1,u_2,u_3)
    -\underline{Q}^{\set,C^*}(\bu).
\end{align*}

Analyzing the summands we see that:
\begin{itemize}
\item $\underline{Q}^{\set,C^*}(\mathbf{s}+\pmb{\varepsilon}) = 
      C^*(\mathbf{s}+\pmb{\varepsilon})$ because 
      $(\mathbf{s}+\pmb{\varepsilon})\in \set$.
\item The expression $\max_{\bx\in \set} \big\{C^*(\mathbf{x})-\sum_{i=1}^3 
      (x_i-v_i)^+\big\}$ where $\bv=(u_1,s_2+\varepsilon_2,s_3+\varepsilon_3)$ 
      attains its maximum either at $\bx=\bs$ or at 
      $\bx=\bs+\pmb{\varepsilon}$, thus equals 
      $\max\{C^*(\mathbf{s}),C^*(\mathbf{s}+\pmb{\varepsilon}) 
      -(s_1+\varepsilon_1-u_1)\}$, while \eqref{req1.4} yields that 
    $C^*(\mathbf{s}+\pmb{\varepsilon})-(s_1+\varepsilon_1-u_1)>C^*(\mathbf{s})$.
      Moreover, \eqref{req1.2} yields $C^*(\bs) \ge W_3(\bs+\pmb{\varepsilon}) 
      \ge W_3(\bv)$, since $\bu\in(\bs,\bs+\pmb{\varepsilon})$. Hence, 
      $$\underline{Q}^{\set,C^*}(u_1,s_2+\varepsilon_2,s_3+\varepsilon_3) = 
      C^*(\mathbf{s}+\pmb{\varepsilon})-(s_1+\varepsilon_1-u_1),$$ 
      while the expressions for the terms involving 
      $(s_1+\varepsilon_1,u_2,s_3+\varepsilon_3)$ and 
      $(s_1+\varepsilon_1,s_2+\varepsilon_2,u_3)$ are analogous.
\item Using the same argumentation, it follows that 
      $$\underline{Q}^{\set,C^*}(u_1,u_2,s_3+\varepsilon_3) = 
      C^*(\mathbf{s}+\pmb{\varepsilon}) - 
      \sum_{i=1,2}(s_i+\varepsilon_i-u_i),$$ 
      while the expressions for the terms involving 
      $(u_1,s_2+\varepsilon_2,u_3)$ and $(s_1+\varepsilon_1,u_2,u_3)$ are 
      analogous.
\item Moreover, $\underline{Q}^{\set,C^*}(\bu) = C^*(\bs)$, which follows from 
      \eqref{req1.2}.
\end{itemize}

Therefore, putting the pieces together and using \eqref{req1.3} we get that
\begin{align*}
V_{\underline{Q}^{\set,C^*}}([\mathbf{u},\mathbf{s}+\pmb{\varepsilon}]) 
  &= C^*(\mathbf{s}+\pmb{\varepsilon}) - 3\, C^*(\mathbf{s}+\pmb{\varepsilon}) 
    + \sum_{i=1}^3(s_i+\varepsilon_i-u_i)\\
  &\quad + 3\, C^*(\mathbf{s}+\pmb{\varepsilon})
    - 2\sum_{i=1}^3(s_i+\varepsilon_i-u_i) - C^*(\mathbf{s})\\
  &= C^*(\mathbf{s}+\pmb{\varepsilon}) - C^*(\mathbf{s}) 
    - \sum_{i=1}^3(s_i+\varepsilon_i-u_i) < 0. 
\end{align*}
Hence $\underline{Q}^{\set,C^*}$ is indeed a proper quasi-copula.
\end{proof}

The following result shows that the requirements in Theorem \ref{lowerBoundIsQuasiCopula} are minimal, in the sense that if the prescription set $\set$ is contained in a set of the form \eqref{minimalS} then the lower bound is indeed a proper quasi-copula.

\begin{corollary}
\label{lowerBoundIsQuasiCopulaCor}
Let $C^*$ be a 3-copula and $\set\subset\mathbb{I}^3$ be compact. 
If there exists a compact set $\set'\supset \set$ such that $\set'$ and $Q^*:=\underline{Q}^{\set,C^*}$ satisfy the assumptions of Theorem \ref{lowerBoundIsQuasiCopula}, then $\underline{Q}^{\set,C^*}$ is a proper quasi-copula.
\end{corollary}

\begin{proof}
Since $Q^*$ and $\set'$ fulfill the requirements of Theorem \ref{lowerBoundIsQuasiCopula}, it follows that $\underline{Q}^{\set',Q^*}$ is a proper quasi-copula.
Now, in order to prove that $\underline{Q}^{\set,C^*}$ is also a proper quasi-copula we will show that $\underline{Q}^{\set',Q^*}=\underline{Q}^{\set,C^*}$. 
Note first that $\underline{Q}^{\set,C^*}$ is the pointwise lower bound of the set $\mathcal{Q}^{\set,C^*}$, i.e.
\begin{align*}
\underline{Q}^{\set,C^*}(\bu) 
  &= \min\big\{Q(\bu)\colon Q\in \mathcal{Q}^{\set,C^*}\big\} 
  = \min\big\{Q(\bu)\colon Q\in\mathcal{Q}^3, Q(\bx) = C^*(\bx) 
    \text{ for all }\bx\in\set\big\}
\end{align*}
for all $\bu\in\mathbb{I}^3$. 
Analogously, $\underline{Q}^{\set',Q^*}$ is the pointwise lower bound of $\mathcal{Q}^{\set',Q^*}$. 
Using the properties of the bounds and the fact that $\set\subset\set'$, it follows that $\underline{Q}^{\set',Q^*}(\bx) = Q^*(\bx) = \underline{Q}^{\set,C^*}(\bx) = C^*(\bx)$ for all $\bx\in\set$. 
Hence, $\underline{Q}^{\set',Q^*}\in\mathcal{Q}^{\set,C^*}$, therefore it holds that $\underline{Q}^{\set,C^*}(\bu)\leq\underline{Q}^{\set',Q^*}(\bu)$ for all 
$\bu\in\mathbb{I}^3$. 
For the reverse inequality, note that for all $\bx\in\set'$ it follows from the definition of $Q^*$ that $\underline{Q}^{\set,C^*}(\bx) = Q^*(\bx)$, hence 
$\underline{Q}^{\set,C^*}\in\mathcal{Q}^{\set',Q^*}$ such that $\underline{Q}^{\set,C^*}(\bu)\geq\underline{Q}^{\set',Q^*}(\bu)$ for all $\bu\in\mathbb{I}^3$.
Therefore, $\underline{Q}^{\set,C^*}=\underline{Q}^{\set',Q^*}$ and $\underline{Q}^{\set,C^*}$ is indeed a proper quasi-copula.
\end{proof}

The next example illustrates Corollary \ref{lowerBoundIsQuasiCopulaCor} in the case where $\set$ is a singleton.

\begin{example}
Let $d=3$, $C^*$ be the independence copula, i.e. $C^*(u_1,u_2,u_3) = u_1u_2u_3$, and $\set = \{\frac12\}^3$. 
Then, the bound $\underline{Q}^{\set,C^*}$ is a proper quasi-copula since its volume is negative, for example 
$V_{\underline{Q}^{\set,C^*}}\big(\big[\frac{5}{10}-\frac{1}{20},\frac{5}{10}\big]^3\big)=-\frac{1}{40}<0$. 
However Theorem \ref{lowerBoundIsQuasiCopula} 
does not apply since $\set$ is not of the form \eqref{minimalS}. 
Nevertheless, using Corollary \ref{lowerBoundIsQuasiCopulaCor}, we can embed $\set$ in a compact set $\set'$ such that $\set'$ and $Q^*:=\underline{Q}^{\set,C^*}$fulfill the conditions of Theorem \ref{lowerBoundIsQuasiCopula}. 
To this end let 
$\set' = \big([0,1]\setminus(s,s+\varepsilon)\big)^3 = \big([0,1]\setminus(\frac{4}{10},\frac{5}{10})\big)^3$, then it follows
\begin{align*}
&\sum_{i=1}^3 \varepsilon = \frac{3}{10} 
  > Q^*\left(\frac{5}{10},\frac{5}{10},\frac{5}{10}\right)
    - Q^*\left(\frac{4}{10},\frac{4}{10},\frac{4}{10}\right) 
  = \left(\frac{5}{10}\right)^3 > 0 \\
\text{and } 
&Q^*\left(\frac{4}{10},\frac{4}{10},\frac{4}{10}\right) 
  = 0 \geq W_3\left(\frac{5}{10},\frac{5}{10},\frac{5}{10}\right) = 0.
\end{align*}
Hence, $Q^*$ and $\set'$ fulfill conditions \eqref{req1.1} and \eqref{req1.2} of Theorem \ref{lowerBoundIsQuasiCopula}, thus it follows from Corollary \ref{lowerBoundIsQuasiCopulaCor} that $\underline{Q}^{\set,C^*}$ is a proper quasi-copula.
\end{example}

\begin{remark}
Analogously to Theorem \ref{lowerBoundIsQuasiCopula} and Corollary \ref{lowerBoundIsQuasiCopulaCor} one obtains that the upper bound $\overline{Q}^{\set,C^*}$ is a proper quasi-copula if the set $\set$ is of the form \eqref{minimalS} and the copula $C^*$ satisfies
\begin{align*}
& \sum_{i=1}^3 
  \varepsilon_i>C^*(\mathbf{s}+\pmb{\varepsilon})-C^*(\mathbf{s})>0 
  \quad\text{and}\quad 
C^*(\mathbf{s}+\pmb{\varepsilon})\leq M_3(\mathbf{s}),
\end{align*}
or if $\set$ is contained in a compact set $\set'$ for which the above hold. 
The respective details and proofs are provided in Appendix \ref{AppB}.
\end{remark}

\section{Stochastic dominance for quasi-copulas}

The aim of this section is to establish a link between the upper and lower orthant order on the set of quasi-copulas and expectations of the associated random variables. 
Let $\mathbf{S} = (S_1,\dots,S_d)$ be an $\mathbb{R}_+^d$-valued random vector with joint distribution ${F}$ and marginals $F_1,\dots,F_d$. 
Using Sklar's theorem, there exists a $d$-copula $C$ such that ${F}(x_1,\dots,x_d) = {C}(F_1(x_1),\dots,F_d(x_d))$ for all $(x_1,\dots,x_d)\in\mathbb{R}_+^d$. Consider a function $f\colon\mathbb{R}_+^d\to\mathbb{R}$; we are interested in the expectation $\mathbb{E}[f(\mathbf{S})]$, in particular in its monotonicity properties with respect to the lower and upper orthant order on the set of quasi-copulas. 
Assuming that the marginals are given, the expectation becomes a function of the copula $C$ and the \textit{expectation operator} is defined via
\begin{align}\label{pi}
\pi_f(C) 
  := \mathbb{E}[f(\mathbf{S})] 
  &= \int_{\mathbb{R}^d} f(x_1,\dots,x_d)\ 
      \ud C(F_1(x_1),\dots,F_d(x_d)) \nonumber\\
  &= \int_{\mathbb{I}^d} f(F_1^{-1}(u_1,),\dots,F_d^{-1}(u_d))\ 
      \ud C(u_1,\dots,u_d).
\end{align}
This definition however no longer applies when $C$ is merely a quasi-copula since the integral in \eqref{pi}, and in particular the term $\ud C$, are no longer well defined. 
This is due to the fact that a quasi-copula $C$ does not necessarily induce a (signed) measure $\ud C$ to integrate against. 
Therefore, we will establish a multivariate integration-by-parts formula which allows for an alternative representation of $\pi_f(C)$ that is suitable for quasi-copulas. 
Similar representations were obtained by \citet{rueschendorf2} for $\Delta$-monotonic functions $f$ fulfilling certain boundary conditions, and by \citet{tankov} for general $\Delta$-monotonic functions $f\colon\mathbb{R}_+^2\to\mathbb{R}$. 
In addition, we will establish properties of the function $f$ such that the extended map $\mathcal{Q}^d\ni Q\mapsto \pi_f(Q)$ is monotonic with respect to the lower and upper orthant order on the set of quasi-copulas.

\citet{rueschendorf2} and \citet{mueller} showed that for $f$ being $\Delta$-antitonic, resp. $\Delta$-mo\-notonic, the map $\mathcal{C}^d\ni C\mapsto \pi_f(C)$ is increasing with respect to the lower, resp. upper, orthant order on the set of copulas. 
$\Delta$-antitonic and $\Delta$-monotonic functions are defined as follows.

\begin{definition}
A function $f\colon\mathbb{R}_+^d\to\mathbb{R}$ is called $\Delta$-\textit{antitonic} if for every subset $\{i_1,\dots,i_n\}\subseteq\{1,\dots,d\}$ with $n\geq 2$ and every hypercube $\times_{j=1}^n[a_j,b_j]\subset\mathbb{R}_+^n$ with $a_j<b_j$ for $j=1,\dots,n$ it holds that 
\[
(-1)^{n}\Delta_{a_1,b_1}^{i_1}\circ\cdots\circ\Delta_{a_n,b_n}^{i_n}\ f(\bx) \geq 0, \quad\text{for all } \bx\in\mathbb{R}_+^d. 
\]
Analogously, $f$ is called $\Delta$-\textit{monotonic} if for every subset $\{i_1,\dots,i_n\}\subseteq\{1,\dots,d\}$ with $n\geq 2$ and every hypercube $\times_{j=1}^n[a_j,b_j]\subset\mathbb{R}_+^n$ with $a_j<b_j$ for $j=1,\dots,n$ it holds
$$\Delta_{a_1,b_1}^{i_1}\circ\cdots\circ\Delta_{a_n,b_n}^{i_n}\ f(\bx)\geq 0, \quad\text{for all } \bx\in\mathbb{R}_+^d.$$
\end{definition}

\begin{remark}
If $f$ is $\Delta$-monotonic then it also $d$-increasing, while if $-f$ is $\Delta$-monotonic then it is also $d$-decreasing.
\end{remark}

As a consequence of Theorem 3.3.15 in \citep*{mueller} we have that for $\underline{C},\overline{C}\in\mathcal{C}^d$ with $\underline{C}\preceq_{LO}\overline{C}$ it follows that $\pi_f(\underline{C})\leq\pi_f(\overline{C})$ for all bounded $\Delta$-antitonic functions $f$. Moreover if $\underline{C}\preceq_{UO}\overline{C}$ it follows that $\pi_f(\underline{C})\leq\pi_f(\overline{C})$ for all bounded $\Delta$-monotonic functions $f$. 

In order to formulate analogous results for the case when $\underline{C},\overline{C}$ are quasi-copulas, let us recall that a function $f:\R^d_+\to\R$ is called \textit{measure inducing} if its volume $V_f$ induces a measure on the Borel $\sigma$-algebra of $\R^d_+$.
Each (componentwise) right-continuous  $\Delta$-monotonic or $\Delta$-antitonic function $f\colon\mathbb{R}_+^d\to\mathbb{R}$ induces a signed measure on the Borel $\sigma$-Algebra of $\mathbb{R}_+^d$, which we denote by $\mu_f$; see Lemma 3.5 and Theorem 3.6 in \citet{gaffke}. 
In particular, it holds that
\begin{align}
\label{inducedMeasure}
\mu_f((a_1,b_1]\times\cdots\times(a_d,b_d])
  = V_f((a_1,b_1]\times\cdots\times(a_d,b_d]),
\end{align}
for every hypercube $(a_1,b_1]\times\cdots\times(a_d,b_d]\subset\mathbb{R}_+^d$. 

Next, we define for a subset $I=\{i_1,\dots,i_n\}\subset\{1,\dots,d\}$ the \textit{$I$-margin} of $f$ via 
$$f_I\colon\mathbb{R}_+^{n}\ni (x_{i_1},\dots,x_{i_n})\mapsto 
  f(x_1,\dots,x_d)\text{ with } x_k = 0\text{ for all } k\notin I,$$
and the associated \textit{$I$-marginal measure} by
$$\mu_{f_I} ((a_{i_1},b_{i_1}]\times\cdots\times(a_{i_n},b_{i_n}]) = 
  V_{f_I}((a_{i_1},b_{i_1}]\times\cdots\times(a_{i_n},b_{i_n}]).$$
Note that if $I=\{1,\dots,d\}$ then $\mu_{f_I}$ equals $\mu_f$, while if $I\subset\{1,\dots,d\}$ then $\mu_{f_I}$ can be viewed as a marginal measure 
of $\mu_f$. 
Now, we define iteratively
\begin{alignat}{2}\label{def:phi}
\text{for } |I|=1: \quad \varphi^I_f(C) &:= && \int_{\mathbb{R}_+} f_{\{i_1\}}(x_{i_1})\ \ud F_{i_1}(x_{i_1}); \nonumber \\ \nonumber
\text{for } |I|=2: \quad 
\varphi^I_f(C) &:= &&-f(0,0)+\varphi^{\{i_1\}}_f(C)+\varphi^{\{i_2\}}_f(C) \nonumber\\
& && +\int_{\mathbb{R}_+^2} \widehat{C_I}(F_{i_1}(x_{i_1}),F_{i_2}(x_{i_2}))\ 	
  \ud\mu_{f_I}(x_{i_1},x_{i_2}); \\ \nonumber
\text{for } |I| =n>2: \quad 
\varphi^I_f(C) &:= && \int_{\mathbb{R}_+^{|I|}} 
  \widehat{C_I}(F_{i_1}(x_{i_1}),\dots,F_{i_n}(x_{i_n}))\ \ud\mu_{f_I}
  (x_{i_1},\dots,x_{i_n})\\\nonumber
& && + \sum_{\substack{J\subset I, J\neq\emptyset}} (-1)^{n+1-|J|} \varphi^J_f(C),
\end{alignat}
where $ \widehat{C_I}$ denotes the survival function of the $I$-margin of $C$. 
The following Proposition shows that $\varphi_f^{\{1,\dots,d\}}$ is an alternative representation of the map $\pi_f$, in the sense that $\pi_f(C) =  \varphi_f^{\{1,\dots,d\}}(C)$ for all copulas $C$.

\begin{proposition}
\label{partialIntegration}
Let $f\colon\mathbb{R}_+^d\to\mathbb{R}$ be measure inducing and $C$ be a $d$-copula. 
Then 
$\pi_f(C) =  \varphi_f^{\{1,\dots,d\}}(C).$
\begin{proof}
Assume first that $f(x_1,\dots,x_d) =  V_f((0,x_1]\times\cdots\times(0,x_d])$ for all $(x_1,\dots,x_d)\in\mathbb{R}_+^d$. 
An application of Fubini's Theorem yields directly that
\begin{align}\label{fubini}
\pi_f(C) &= \int_{\mathbb{R}^d_+} f(x_1,\dots,x_d)\ 
  \ud C(F_1(x_1),\dots,F_d(x_d)) \nonumber \\
& = \int_{\mathbb{R}^d_+}  V_f((0,x_1]\times\cdots\times(0,x_d])\ 
  \ud C(F_1(x_1),\dots,F_d(x_d)) \nonumber \\
& =  \int_{\mathbb{R}^d_+}  \mu_f((0,x_1]\times\cdots\times(0,x_d])\ 
  \ud C(F_1(x_1),\dots,F_d(x_d)) \nonumber \\
& = \int_{\mathbb{R}^d_+}\bigg(\int_{\mathbb{R}^d_+} 
  \b1_{x_1'<x_1}\cdots\b1_{x_d'<x_d} 
  \ud\mu_f(x_1',\dots,x_d')\bigg)\ 
  \ud C(F_1(x_1),\dots,F_d(x_d)) \nonumber \\
& = \int_{\mathbb{R}^d_+}\bigg(\int_{\mathbb{R}^d_+} 
  \b1_{x_1'>x_1}\cdots\b1_{x_d'>x_d}\ud C(F_1(x_1'),\dots,
  F_d(x_d'))\bigg)\ \ud\mu_f(x_1,\dots,x_d) \nonumber \\
& = \int_{\mathbb{R}^d_+} \widehat{C}(F_1(x_1),\dots,F_d(x_d))\ 
  \ud\mu_f(x_1,\dots,x_d),
\end{align}
where the last equality follows from \eqref{survivalCopulaProbability}. 
Next, we drop the assumption $f(x_1,\dots,x_d)=V_f((0,x_1]\times\cdots\times(0,x_d])$ and show that the general statement holds by induction over the dimension $d$. 
By Proposition 2 in \citep*{tankov} we know that the statement is valid for $d=2$. 
Now, assume it holds true for $d=n-1$, then for $d=n$ we have that 
$$f(x_1,\dots,x_n) = V_f((0,x_1]\times\cdots\times(0,x_n]) - 
  [V_f((0,x_1]\times\cdots\times(0,x_n])-f(x_1,\dots,x_n)].$$
Noting that $V_f((0,x_1]\times\cdots\times(0,x_n])-f(x_1,\dots,x_n)$ is a sum of functions each with domain $\mathbb{R}^k_+$ with $k\leq n-1$, it follows
\begin{alignat*}{2}
\pi_f(C) &= && \int f(x_1,\dots,x_n)\ \ud C(F_1(x_1),\dots,F_n(x_n)) \\
& = && \int V_f((0,x_1]\times\cdots\times(0,x_n])\ \ud 
C(F_1(x_1),\dots,F_n(x_n))\\
& &&- \int [V_f((0,x_1]\times\cdots\times(0,x_n])-f(x_1,\dots,x_n)]\ 
\ud C(F_1(x_1),\dots,F_n(x_n))\\
&= &&\int \widehat{C}(F_1(x_1),\dots,F_n(x_n))\ \ud\mu_f(x_1,\dots,x_n)\\
& && + \int [-V_f((0,x_1]\times\cdots\times(0,x_n])+f(x_1,\dots,x_n)]\ 
\ud C(F_1(x_1),\dots,F_n(x_n))\\
&= &&\int \widehat{C}(F_1(x_1),\dots,F_n(x_n))\ \ud\mu_f(x_1,\dots,x_n) + 
\sum_{\substack{J\subset \{1,\dots,n\}\\J\neq\emptyset}} (-1)^{n+1-|J|} \varphi^J_f(C),
\end{alignat*}
where we have applied equation \eqref{fubini} to $\int V_f((0,x_1]\times\cdots\times(0,x_n]) \ \ud C(F_1(x_1),\dots,F_n(x_n))$ to obtain the third equality, and used the induction hypothesis for the last equality as for each $J\subset \{1,\dots,n\}$ the domain of $f_J$ is $\mathbb{R}^{|J|}$ with $|J|\leq n-1$. 
\end{proof}
\end{proposition}

Proposition \ref{partialIntegration} enables us to extend the notion of the expectation operator $\pi_f$ to quasi-copulas and establish monotonicity properties for the generalized mapping. 

\begin{definition}\label{def:qeo}
Let $f\colon\mathbb{R}^d_+\to\mathbb{R}$ be measure inducing. 
Then, the \textit{quasi-expectation operator} for $Q \in \mathcal{Q}^d$ is defined via
$$\pi_f(Q) :=  \int_{\mathbb{R}_+^d} \widehat{Q}(F_1(x_1),\dots,F_d(x_d)) \ 
  \ud\mu_f(x_1,\dots,x_d) + \sum_{\substack{J\subset \{1,\dots,d\}\\ J\neq\emptyset}} (-1)^{d+1-|J|} 
  \varphi^J_f(Q).$$
\end{definition}

\begin{theorem}
\label{extensionOrder}
Let $\underline{Q},\overline{Q}\in\mathcal{Q}^d$, then it holds 
\begin{align*}
&(i) \quad 
\underline{Q} \preceq_{LO} \overline{Q} \quad \Longrightarrow \quad 
  \pi_f(\underline{Q}) \leq \pi_f(\overline{Q}) 
&&\text{for all }\Delta\text{-antitonic } f\colon\mathbb{R}_+^d\to\mathbb{R}\\
& &&\text{s.t. the integrals exist;} \\
&(ii) \quad 
\underline{Q} \preceq_{UO} \overline{Q} \quad \Longrightarrow \quad 
  \pi_f(\underline{Q})\leq\pi_f(\overline{Q})
&&\text{for all }\Delta\text{-monotonic } f\colon\mathbb{R}_+^d\to\mathbb{R}\\
& &&\text{s.t. the integrals exist.}
\end{align*}
Moreover, if $F_1,\dots,F_d$ are continuous then the converse statements are also true. 
\end{theorem}

\begin{proof}
We prove the statements assuming that the condition $f(x_1,\dots,x_d) = V_f((0,x_1]\times\cdots\times(0,x_d])$ holds. 
The general case follows then by induction as in the proof of Proposition \ref{partialIntegration}. 
Let $f$ be $\Delta$-antitonic and $\underline{Q}\preceq_{LO}\overline{Q}$, then it follows
\begin{align*}
\pi_f(\underline{Q}) 
&= \int_{\mathbb{R}_+^d} \widehat{\underline{Q}}(F_1(x_1),\dots,F_d(x_d))\ 
  \ud\mu_f(x_1,\dots,x_d)\\
&= \int_{\mathbb{R}_+^d} 
  V_{\underline{Q}}\big((F_1(x_1),1]\times\cdots\times(F_d(x_d),1]\big)\ 
  \ud\mu_f(x_1,\dots,x_d)\\
&= \int_{\mathbb{R}_+^d} \Big\{ \underline{Q}(1,\dots,1) 
  - \underline{Q}(F_1(x_1),1,\dots,1) - \dots  
  - \underline{Q}(1,\dots,1,F_d(x_d))\\ 
&\qquad + \underline{Q}(F_1(x_1),F_2(x_2),1,\dots,1) + \cdots 
  + \underline{Q}(1,\dots,1,F_{d-1}(x_{d-1}),F_d(x_d)) \\
&\qquad  - \cdots + (-1)^d \underline{Q}(F_1(x_1),\dots.,F_d(x_d))\Big\}\ 
  \ud\mu_f(x_1,\dots,x_d)\\
&= \int_{\mathbb{R}_+^d} \Big\{ \underline{Q}(1,\dots,1) 
  + \underline{Q}(F_1(x_1),1,\dots,1) + \dots  
  + \underline{Q}(1,\dots,1,F_d(x_d))\\ 
&\qquad + \underline{Q}(F_1(x_1),F_2(x_2),1,\dots,1) + \cdots 
  + \underline{Q}(1,\dots,1,F_{d-1}(x_{d-1}),F_d(x_d)) \\
&\qquad  + \cdots + \underline{Q}(F_1(x_1),\dots.,F_d(x_d))\Big\}\ 
  \ud|\mu_f|(x_1,\dots,x_d),
\end{align*}
where for the last equality we used that $f$ is $\Delta$-antitonic, hence $\mu_f$ has alternating signs. 
A similar representation holds for $\pi_f(\overline{Q})$, thus
\begin{align*}
&\pi_f(\overline{Q}) - \pi_f(\underline{Q}) \\
&\quad = \int_{\mathbb{R}^d} \Big\{ \big[ \overline{Q}(F_1(x_1),1,\dots,1)  - 
  \underline{Q}(F_1(x_1),1,\dots,1) \big] + \cdots \\
&\qquad\qquad + \big[ \overline{Q}(1,\dots,1,F_d(x_d) - 
  \underline{Q}(1,\dots,1, F_d(x_d))\big] \\
&\qquad\qquad + \big[ \overline{Q}(F_1(x_1),F_2(x_2),1,\dots,1)
  - \underline{Q}(F_1(x_1),F_2(x_2),1,\dots,1) \big] + \cdots\\
&\qquad + \big[ \overline{Q}(1,\dots,1,F_{d-1}(x_{d-1}),F_d(x_d)) 
  - \underline{Q}(1,\dots,1,F_{d-1}(x_{d-1}),F_d(x_d))\big] + \cdots\\
&\qquad + \big[ \overline{Q}(F_1(x_1),\dots.,F_d(x_d))
  - \underline{Q}(F_1(x_1),\dots,F_d(x_d))\Big\}\ 
  \ud|\mu_f|(x_1,\dots,x_d) \ \geq 0,
\end{align*}
since $\underline{Q}\preceq_{LO}\overline{Q}$. 
Hence assertion (i) is true. 
Regarding (ii), we have directly that
\begin{align*}
\pi_f(\overline{Q}) &- \pi_f(\underline{Q})\\
&= \int_{\mathbb{R}^d} \Big\{ \widehat{\overline{Q}}(F_1(x_1),\dots,F_d(x_d)) 
  - \widehat{\underline{Q}}(F_1(x_1),\dots,F_d(x_d)) \Big\}\ 
  \ud\mu_f(x_1,\dots,x_d) \geq 0,
\end{align*}
where we used that $f$ is $\Delta$-monotonic, hence $\mu_f$ is a positive measure, as well as  $\underline{Q}\preceq_{UO}\overline{Q}$.

As for the converse statements, assume that $F_1,\dots,F_d$ are continuous. 
If $\pi_f(\underline{Q})\leq\pi_f(\overline{Q})$ holds for all $\Delta$-antitonic $f$, then it holds in particular for functions of the form $f(x_1,\dots,x_d) = \b1_{x_1\leq u_1,\dots,x_d \leq u_d}$, for arbitrary $(u_1,\dots,u_d) \in(0,\infty]^d$. 
For such $f$ and any quasi-copula $Q$ it holds that $\pi_f(Q) = Q(F_1(u_1),...,F_d(u_d))$; cf. \citet[Lemma~3.1.4]{Lux_2017}.
Hence
\begin{align*}
\pi_f(\underline{Q}) \leq \pi_f(\overline{Q}) 
  \quad \Longrightarrow \quad 
\underline{Q}(F_1(u_1),\dots,F_d(u_d)) \leq 
  \overline{Q}(F_1(u_1),\dots,F_d(u_d)),
\end{align*}
while from the fact that $\pi_f(\underline{Q})\leq\pi_f(\overline{Q})$ holds  for all choices of $(u_1,\dots,u_d)$ and the continuity of the marginals it 
follows that (i) holds. 
Assertion (ii) follows by an analogous argument. 
Note that if $\pi_f(\underline{Q})\leq\pi_f(\overline{Q})$ holds for all $\Delta$-monotonic $f$, it holds in particular for functions of the form $f(x_1,\dots,x_d) = \b1_{x_1\geq u_1,\dots,x_d \geq u_d}$ for arbitrary $(u_1,\dots,u_d) \in (0,\infty]^d$. 
For such $f$ and any quasi-copula $Q$, it holds that $\pi_f(Q) = \widehat{Q}(F_1(u_1),...,F_d(u_d))$ and so (ii) follows as above.
\end{proof}

\begin{remark}\label{extensionOrder-E}
Consider the setting of Theorem \ref{extensionOrder} and assume that $-f$ is $\Delta$-antitonic, resp. $\Delta$-monotonic.
Then, the inequalities on the right hand side of ($i$) and ($ii$) are reversed, i.e.
\[
\underline{Q} \preceq_{LO} \overline{Q} \ \Longrightarrow \ \pi_f(\underline{Q}) \geq \pi_f(\overline{Q})
\quad \text {and} \quad 
\underline{Q} \preceq_{UO} \overline{Q} \ \Longrightarrow \ \pi_f(\underline{Q}) \geq \pi_f(\overline{Q}).
\]
\end{remark}
%
%
%
\begin{remark}
Let us point out that the class of $\Delta$-antitonic functions is the maximal generator of the lower orthant order on the set of copulas, i.e. every $f\colon\mathbb{R}_+^d\to\mathbb{R}$ such that 
$$\underline{C}\preceq_{LO} \overline{C} \quad\Longrightarrow\quad \pi_f(\underline{C}) \leq \pi_f(\overline{C})$$
is $\Delta$-antitonic; see \citep[Theorem 3.3.15]{mueller}. 
Hence, statement (i) in the theorem above cannot be further weakened. 
Conversely, the set of $\Delta$-monotonic functions is the maximal generator of the upper orthant order, thus statement (ii) in the theorem can also not be further relaxed.  
\end{remark}

Finally, we provide an integrability condition for the extended map $\pi_f(\cdot)$ based on the marginals $F_1,\dots,F_d$ and the properties of the function $f$.
In particular, the finiteness of $\pi_f(C)$ is independent of $C$ being a copula or a proper quasi-copula.

\begin{proposition}
\label{iteration}
Let $f\colon\mathbb{R}_+^d\to\mathbb{R}$ be right-continuous, $\Delta$-antitonic or $\Delta$-monotonic such that
\begin{align}
\label{requirementFiniteExpectation}
\sum_{J\subset \{1,\dots,d\}}\ \sum_{i=1}^d \Bigg\{ \int_{\mathbb{R}_+^{|J|}} 
  |f_J(x,\dots.,x)|\ \ud F_i(x)\Bigg\} < \infty.
\end{align}
Then the map $\pi_f$ is well-defined and continuous with respect to the pointwise convergence of quasi-copulas.
\end{proposition}

\begin{proof}
First, we show that for $C\in\mathcal{C}^d$ the expectation $\int f(x_1,\dots,x_d)\ \ud C(F_1(x_1),\dots,F_d(x_d))$ is finite by induction over the dimension $d$. 
By Proposition 2 in \citep*{tankov} we know that the statement is true for $d=2$. 
Assume that the statement holds for $d=n-1$, then for $d=n$ we have
\begin{align}
\label{requirementFiniteExpectationInequality}
|f(x_1,&\dots,x_n)| \nonumber \\ \nonumber
&= |V_f((0,x_1]\times\cdots\times(0,x_n]) - 
  (V_f((0,x_1]\times\cdots\times(0,x_n]) - f(x_1,\dots,x_n)| \\ \nonumber
&\leq |V_f((0,x_1]\times\cdots\times(0,x_n])| 
  + |V_f((0,x_1]\times\cdots\times(0,x_n]) - f(x_1,\dots,x_n)| \\ \nonumber
&\leq |V_f((0,x_1]^n)| + \cdots + |V_f((0,x_n]^n)| 
  + |V_f((0,x_1]\times\cdots\times(0,x_n] ) - f(x_1,\dots,x_n)| \\ \nonumber
&\leq \sum_{i=1}^{n} \sum_{J\subset \{1,\dots,n\}} |f_J(x_i,\dots,x_i)| 
  + |V_f((0,x_1]\times\cdots\times(0,x_n])-f(x_1,\dots,x_n)| \\
&\leq \sum_{i=1}^{n} \sum_{J\subset \{1,\dots,n\}} |f_J(x_i,\dots,x_i)| 
  + \text{const } \cdot \sum_{J\subset \{1,\dots,n\}}   
|f_J(x_1,\dots,x_n)|,
\end{align}
where the second inequality follows from the definition of $V_f$ and $\times_{i=1}^n(0,x_i]\subseteq \bigcup_{i=1}^n (0,x_i]^n$. 
Now, note that for $J\subset \{1,\dots,n\}$ $f$ is a function with domain $\mathbb{R}_+^{|J|}$ where $|J|<n$, hence by the induction hypothesis and \eqref{requirementFiniteExpectation} we get that 
$$\int_{\mathbb{R}_+^{|J|}} |f_J(x_1,\dots.,x_n)|\ 
	\ud C_J(F_1(x_1),\dots,F_n(x_n))<\infty$$
for each $J\subset\{1,\dots,n\}$, where $|J|\leq n-1$. 
Hence
$$b := \text{const }\cdot \sum_{J\subset 
  \{1,\dots,n\}}\Bigg\{\int_{\mathbb{R}_+^{|J|}} |f_J(x_1,\dots.,x_n)|\ 
  \ud C_J(F_1(x_1),\dots,F_n(x_n))\Bigg\} < \infty.$$
Finally, from \eqref{requirementFiniteExpectation} and \eqref{requirementFiniteExpectationInequality} we obtain
\begin{align*}
\int_{\mathbb{R}_+^n} |f(x_1,\dots,x_n)|\ & \ud C(F_1(x_1),\dots,F_n(x_n)) \\
&\leq \sum_{J\subset \{1,\dots,n\}} \ \sum_{i=1}^n 
  \Bigg\{\int_{\mathbb{R}_+^{|J|}} |f_J(x,\dots.,x)|\ \ud F_i(x)\Bigg\} + b
<\infty.
\end{align*}
Hence the assertion is true for $\mathcal{C}^d\ni C\mapsto\pi_f(C)$. 
Now for the extended map, let $Q$ be a proper quasi-copula and assume that $f$ is $\Delta$-antitonic. 
Then it follows from Theorem \ref{extensionOrder} and the properties of the upper \FH bound that $0\leq\pi_f(Q)\leq\pi_f(M_d)<\infty$, where the finiteness of $\pi_f(M_d)$ follows from the fact that $M_d\in\mathcal{C}^d$. 
By the same token, since all quasi-copulas are bounded from above by the upper \FH bound $M_d$ and the integrals with respect to $M_d$ exist, the dominated convergence theorem yields that $\pi_f$ is continuous with respect to the pointwise convergence of quasi-copulas. 
The well-definedness of $\pi_f$  for $\Delta$-monotonic $f$ follows analogously. 
\end{proof}

\section{Applications in model-free finance}

A direct application of our results is the computation of bounds on the prices of multi-asset options assuming that the marginal distributions of the assets are fully known while the dependence structure between them is only partially known. 
This situation is referred to in the literature as \textit{dependence uncertainty} and the resulting bounds as \textit{model-free bounds} for the option prices.
The literature on model-free bounds for multi-asset option prices focuses almost exclusively on basket options, see e.g. \citet{Hobson_Laurence_Wang_2005_1,Hobson_Laurence_Wang_2005_2}, \citet{dAspremont_ElGhaoui_2006}, \citet{Chen_Deelstra_Dhaene_Vanmaele_2008} and \citet{Pena_Vera_Zuluaga_2010}, while \citet{tankov} considers general payoff functions in a two-dimensional setting.
See also \citet{Dhaene_etal_2002_a,Dhaene_etal_2002_b} for applications of model-free bounds in actuarial science.

We consider European-style options whose payoff depends on a positive random vector $\mathbf{S} = (S_1,\dots,S_d)$. 
The constituents of $\mathbf{S}$ represent the values of the option's underlyings at the time of maturity. 
In the absence of arbitrage opportunities, the existence of a risk-neutral probability measure $\mathbb{Q}$ for $\mathbf{S}$ is guaranteed by the fundamental theorem of asset pricing. 
Then, the price of an option on $\mathbf{S}$ equals the discounted expectation of its payoff under a risk-neutral probability measure. 
We assume that all information about the risk-neutral distribution of $\mathbf{S}$ or its constituents comes from prices of traded derivatives on these assets, and that single-asset European call options with payoff $(S_i-K)^+$ for $i=1,\dots,d$ and for all strikes $K>0$ are liquidly traded in the market. 
Assuming zero interest rates, the prices of these options are given by $\Pi^i_K = \mathbb{E}_{\mathbb{Q}}[(S_i-K)^+]$. 
Using these prices, one can fully recover the risk neutral marginal distributions $F_i$ of $S_i$ as shown by \citet{breeden}.

Let $f\colon\mathbb{R}_+^d\to\mathbb{R}$ be the payoff of a European-style option on $\mathbf{S}$. 
Given the marginal risk-neutral distributions $F_1,\dots,F_d$ of $S_1,\dots,S_d$, the price of $f(\mathbf{S})$ becomes a function of the copula $C$ of $\mathbf{S}$ and is provided by the expectation operator as defined in \eqref{pi}, i.e.
$$
\mathbb{E}_{\mathbb{Q}}[f(S_1,\dots,S_d)] = \pi_f(C). 
$$
Assuming that the only available information about the risk-neutral distribution of $\mathbf{S}$ is the marginal distributions, the set of all arbitrage-free prices for $f(\mathbf{S})$ equals $\Pi := \{\pi_f(C)\colon C\in\mathcal{C}^d\}$. 
Moreover, if additional information on the copula $C$ is available, one can narrow the set of arbitrage-free prices by formulating respective constraints on the copula. 
Let therefore $\mathcal{C}^*$ represent any of the constrained sets of copulas from Section \ref{sec:iFH} or Appendix \ref{AppA}, and define the set of arbitrage-free prices compatible with the respective constraints via  $\Pi^* := \{\pi_f(C)\colon C\in \mathcal{C}^*\}$. 
Since $\mathcal{C}^*\subset\mathcal{C}$ we have immediately that $\Pi^*\subset\Pi$.

Theorem \ref{extensionOrder} yields that if the payoff $f$ is $\Delta$-antitonic, then $\pi_f(C)$ is monotonically increasing in $C$ with respect to the lower orthant order. 
Conversely if $f$ is $\Delta$-monotonic, then $\pi_f(C)$ is monotonically increasing in $C$ with respect to the upper orthant order. 
In the following result, we exploit this fact to compute bounds on the sets $\Pi$ and $\Pi^*$. 
Let us first define the dual $\widehat{\pi}$ of the operator $\pi$ on the set of survival functions, via $\widehat{\pi}(\widehat{C}) := \pi(C)$. 

\begin{proposition}
\label{noInformationBounds}
Let $f$ be $\Delta$-antitonic and $\underline{Q}^*,\overline{Q}^*\in\mathcal{Q}^d$ be a lower and an upper bound on the constrained set of copulas $\mathcal{C}^*$ with respect to the lower orthant order. 
Then 
$$
\pi_f(W_d) \leq \pi_f(\underline{Q}^*) \leq \inf\Pi^* \leq \pi_f(C) 
  \le \sup\Pi^* \leq \pi_f(\overline{Q}^*) \leq \pi_f(M_d) = \sup\Pi
$$
for all $C \in \mathcal{C}^*$, in case the respective integrals exist, while $\inf\Pi=\pi_f(W_d)$ if $d=2$. 
In this setting, if $-f$ is $\Delta$-antitonic, then all inequalities in the above equation are reversed.

Moreover, if $f$ is $\Delta$-monotonic, $\mathcal{C}^*$ is a constrained set of copulas and $\underline{Q}^*, \overline{Q}^*\in\mathcal Q^d$ are a lower and an upper bound on $\mathcal C^*$ with respect to the upper orthant order, then
$$
\widehat{\pi}_f(W_d(\mathbf 1 \!-\! \cdot)) \leq \widehat{\pi}_f(\underline{\widehat{Q}}^*) 
  \leq \inf\Pi^* \leq \pi_f(C) \le \sup\Pi^* 
  \leq \widehat{\pi}_f(\widehat{\overline{Q}}^*) \leq \widehat{\pi}_f(M_d(\mathbf 1 \!-\! \cdot)) 
  = \sup\Pi
$$
for all $C \in \mathcal{C}^*$, if the respective integrals exist, while $\inf\Pi=\widehat{\pi}_f(W_d(\mathbf 1-\cdot))=$ holds if $d=2$.
In this setting, if $-f$ is $\Delta$-monotonic, then all inequalities in the equation above are reversed.
\end{proposition}

\begin{proof}
Let $C \in \mathcal{C}^*$, then it holds that
\[
W_d \preceq_{LO} \underline{Q}^* \preceq_{LO} C \preceq_{LO} \overline{Q}^* 
  \preceq_{LO} M_d,
\]
and the result follows from Theorem \ref{extensionOrder}(i) for a $\Delta$-antitonic function $f$. 
Note that $\sup\Pi = \pi_f(M_d)$ since the upper \FH bound is again a copula. 
The second statement follows analogously from the properties of the improved \FH bounds on survival functions, which are provided in Appendix \ref{AppA}, and an application of Theorem \ref{extensionOrder}(ii).
The statements for $-f$ being $\Delta$-antitonic or $\Delta$-monotonic follow using the same arguments combined with Remark \ref{extensionOrder-E}.
\end{proof}

\begin{remark}
Let us point out that $\pi_f(M_d)$ is an upper bound on the set of prices $\Pi$ even under weaker assumptions on the payoff function $f$ than $\Delta$-motonocity or $\Delta$-antitonicity.
This is due to the fact that the upper \FH bound is a copula, thus a sharp bound on the set of all copulas. 
\citet{Hobson_Laurence_Wang_2005_1}, for example, derived upper bounds on basket options and showed that these bounds are attained by a comonotonic random vector having copula $M_d$. 
Moreover, \citet{carlier2003} obtained bounds on $\Pi$ for $f$ being monotonic of order 2 using an optimal transport approach. 
He further showed that these bounds are attained for a monotonic rearrangement of a random vector, which in turn leads to the upper \FH bound. 
\end{remark}

\begin{remark}
Let $\underline{Q}^*$ be any of the improved \FH bounds from Section \ref{sec:iFH}.
Then, the inequality 
\begin{align}\label{sharpBetterThanImproved}
\inf\Pi\leq\pi_f(\underline{Q}^*)
\end{align}
does not hold in general. 
In particular, the sharp bound $\inf\Pi$ without additional dependence information might exceed the price bound obtained using $\underline{Q}^*$. 
A sufficient condition for \eqref{sharpBetterThanImproved} to hold is the existence of a copula $C\in\mathcal{C}^d$ such that $C\leq\underline{Q}^*$. 
This condition is however difficult to verify in practice.
In many cases $\inf\Pi$ cannot be computed analytically, hence a direct comparison of the bounds is usually not possible. 
On the other hand, one can resort to computational approaches in order to check whether \eqref{sharpBetterThanImproved} is satisfied. 
A numerical method to compute $\inf\Pi$ for continuous payoff functions $f$ fulfilling a minor growth condition, based on the assignment problem, is presented in \citet{preischl2016}. 
This approach thus lends itself to a direct comparison of the bounds. 
\end{remark}

Let us recall that by Proposition \ref{partialIntegration} the computation of $\pi_f$ amounts to an integration with respect to the measure $\mu_f$ that is 
induced by the function $f$. 
The following table provides some examples of measure inducing payoff functions $f$ along with explicit representations of the integrals with respect to $\mu_f$.
More specifically, for a $\Delta$-motononic or $\Delta$-antitonic function $f$, the expression $\int g(x_{i_1},\dots,x_{i_n})\ \ud\mu_{f_I}$ refers to the summands of $\pi_f$ for $I = \{i_1,\dots,i_n\}$; see again Definition \ref{def:qeo} and \eqref{def:phi}. 
An important observation here is that the multi-dimensional integrals with respect to the copula reduce to one-dimensional integrals with respect to the induced measure, which makes the computation of option prices very fast and efficient. 

\begin{table}[H]
\begin{center}
\begin{tabular}{lcl} 
\hline \hline
Payoff $f(x_1,\dots,x_d)$ 
	& \hspace{0.25cm} $\Delta$-tonicity \hspace{0.25cm} 
	& $\int g(x_{i_1},\dots,x_{i_n})\ \ud\mu_{f}$ \\ \hline \\[-.45em]
\shortstack[l]{Digital put on maximum \\ $\mathds{1}_{\max\{x_1,\dots,x_d\}\leq K}$} 
	& $f$ antitonic 
	& $\begin{cases} g(K,\dots,K), & |I|\text{ even}\\ -g(K,\dots,K), & |I|\text{ odd}\end{cases}$ \\ \\
\shortstack[l]{Digital call on minimum \\ $\mathds{1}_{\min\{x_1,\dots,x_d\}\geq K}$} 
	& $f$ monotonic
	& $\begin{cases} g(K,\dots,K), & I = \{1,\dots,d\}\\ 0, & \text{else}\end{cases}$ \\ \\ 
\shortstack[l]{Call on minimum \\ $(\min\{x_1,\dots,x_d\}-K)^+$} 
	& $f$ monotonic
	& $\begin{cases}\int_K^\infty g(x,\dots,x) \ud x, &I = \{1,\dots,d\}\\ 0, & \text{else}\end{cases}$ \\ \\
\shortstack[l]{Put on minimum \\$(K-\min\{x_1,\dots,x_d\})^+$} 
	& $-f$ monotonic
	& $\begin{cases}\int_0^K g(x,\dots,x) \ud x, &I = \{1,\dots,d\}\\ 0, & \text{else}\end{cases}$ \\\\  
\shortstack[l]{Call on maximum \\$(\max\{x_1,\dots,x_d\}-K)^+$} 
	& $-f$ antitonic
	& $\begin{cases}-\int_K^\infty g(x,\dots,x) \ud x, &|I|\text{ even}\\
		\int_K^\infty g(x,\dots,x) \ud x, &   |I|\text{ odd}\end{cases}$ \\ \\
\shortstack[l]{Put on maximum \\$(K-\max\{x_1,\dots,x_d\})^+$} 
	& $f$ antitonic
	& $\begin{cases}\int_0^K g(x,\dots,x) \ud x, &|I|\text{ even}\\ 
		-\int_0^K g(x,\dots,x) \ud x &   |I|\text{ odd}\end{cases}$ 
		\\\\[-.5em]
\hline \hline
\end{tabular}
\caption{Examples of payoff functions for multi-asset options and the respective representation of the integral with respect to the measure $\mu_f$. 
			The formulas for the digital call on the maximum and the digital put on the minimum can be obtained by a put-call parity.}
\label{tab:payoff-rep}
\end{center}
\end{table}

\begin{remark}[Differentiable payoffs]
Assume that the payoff function is differentiable, i.e. the partial derivatives of the function $f$ exist. 
Then, we obtain the following representation for the integral with respect to $\mu_f$:
$$
\int_{\mathbb{R}_+^d} g(x_1,\dots,x_d)\ \ud\mu_f(x_1,\dots,x_d) 
  = \int_{\mathbb{R}_+^d} g(x_1,\dots,x_d)\ \frac{\partial^d f(x_1,\dots,x_d)} 
    {\partial x_1\cdots\partial x_d}\ \ud x_1\cdots \ud x_d.
$$
The formula holds, because from the definition of the volume $V_f$ we get that
$$
V_f(H) = \int_{H} \frac{\partial^d f(x_1,\dots,x_d)}
	{\partial x_1\cdots\partial x_d}\ \ud x_1\cdots \ud x_d,
$$
for every $H$-box in $\mathbb{R}^d_+$.
Differentiable $\Delta$-antitonic functions occur in problems related to utility maximization; see, for example, the definition of Mixex utility functions in \citet*{tsetlin2009}.
\end{remark}

\begin{remark}[Basket and spread options]
Although basket options on two underlyings are $\Delta$-monotonic, their higher-dimensional counterparts, i.e. $f\colon\mathbb{R}_+^d\ni(x_1,\dots,x_d) \mapsto \big(\sum_{i=1}^d \alpha_ix_i-K\big)^+$ for $\alpha_i,\dots,\alpha_d\in\mathbb{R}_+$, are neither $\Delta$-monotonic nor $\Delta$-antitonic in general. 
However, from the monotonicity of bivariate basket options it follows that their expectation is monotonic with respect to the lower and upper orthant order on the set of 2-copulas. 
Therefore, prices of bivariate basket options provide information that can be accounted for by Theorems \ref{prescribedFunctional} or \ref{regionalPrescription}. 
In particular, if $f\colon\mathbb{R}_+^2\ni(x_1,x_2)\mapsto(\alpha_1x_1+\alpha_2x_2-K)^+$ then $f$ is $\Delta$-monotonic for $\alpha_1\alpha_2> 0$, thus $\rho(C):=\pi_f(C)$ is increasing with respect to the lower and upper orthant order on $\mathcal{C}^2$. 
Analogously, if $f$ is a spread option, i.e. $\alpha_1\alpha_2< 0$, then $\rho(C):=-\pi_f(C)$ is increasing with respect to the lower and upper orthant order on $\mathcal{C}^2$. 
Thus, by means of Theorem \ref{prescribedFunctional} one can translate market prices of basket or spread options into improved \FH bounds for 2-copulas which may then serve as information to compute higher-dimensional bounds by means of Theorem \ref{regionalPrescription}.
\end{remark}

An interesting question arising naturally is under what conditions the bounds in Proposition \ref{noInformationBounds} are \textit{sharp}, in the sense that 
\begin{align}
\label{sharpnessOfExpectation}
\inf\Pi^* = \pi_f(\underline{Q}^*)
  \quad \text{ and } \quad 
\sup\Pi^* = \pi_f(\overline{Q}^*),
\end{align}
and similarly for $\pi_f(\underline{\widehat{Q}}^*)$ and $\pi_f(\widehat{\overline{Q}}^*)$. 
In Section \ref{boundsQuasiCopulas} we showed that the improved \FH bounds fail to be copulas, hence they are not sharp in general. 
However, by introducing rather strong conditions on the function $f$, we can obtain the sharpness of the integral bounds in the sense of \eqref{sharpnessOfExpectation} when $\underline{Q}^*$ and $\overline{Q}^*$ are the improved \FH bounds. 
In order to formulate such conditions we introduce the notion of an \textit{increasing $d$-track} as defined by \citet{genest}.

\begin{definition}
\label{dTrack}
Let $G_1,\dots,G_d$ be continuous, univariate distribution functions on $\overline{\mathbb{R}}$, such that $G_i(-\infty) = 0$ and $G_i(\infty) = 1$ for $i=1,\dots,d$. 
Then, $T^d:=\{(G_1(x),\dots,G_d(x))\colon x\in\overline{\mathbb{R}}\}\subset\mathbb{I}^d$ is an (increasing) $d$-track in $\mathbb{I}^d$.
\end{definition}

The following result establishes sharpness of the option price bounds, under conditions which are admitedly rather strong for practical applications.

\begin{proposition}
\label{sharpBounds}
Let $f\colon\mathbb{R}_+^d\to\mathbb{R}$ be a right-continuous, $\Delta$-monotonic function that satisfies $f(x_1,\dots,x_d) = V_f([0,x_1]\times\cdots\times[0,x_d])$. 
Assume that
$$
\mathcal{B} := \{ (F_1(x_1),\dots,F_d(x_d))\colon\mathbf{x}\in \operatorname{supp}\mu_f \} \subset T^d,
$$ 
for some $d$-track $T^d$. 
Moreover, consider the upper and lower bounds $\underline{\widehat{Q}}^{\set,C^*},\widehat{\overline{Q}}^{\set,C^*}$ from Corollary \ref{PrescriptionOnSubsetSurvival}. 
Then, if $\set\subset T^d$ it follows that 
\begin{align*}
\inf\big\{\pi_f(C)\colon C\in\widehat{\mathcal{C}}^{\set,C^*}\big\} 
  = \widehat{\pi}\Big(\underline{\widehat{Q}}^{\set,C^*}\Big)
\quad\text{and}\quad 
\sup\big\{\pi_f(C)\colon C\in\widehat{\mathcal{C}}^{\set,C^*}\big\} 
  = \widehat{\pi}\Big(\widehat{\overline{Q}}^{\set,C^*}\Big).
\end{align*}
\end{proposition}

\begin{proof}
Since $\mathbf{u}\mapsto\underline{\widehat{Q}}^{\set,C^*}(\mathbf1-\bu)$ and $\mathbf{u}\mapsto\widehat{\overline{Q}}^{\set,C^*}(\mathbf1-\bu)$ are quasi-copulas and $\mathcal B$ is a subset of a $d$-track $T^d$, it follows from the properties of a quasi-copula, see \citet{lallena}, that there exist survival copulas $ \underline{\widehat{C}}^{\set,C^*}$ and $\widehat{\overline{C}}^{\set,C^*}$ which coincide with $\underline{\widehat{Q}}^{\set,C^*}$ and $\widehat{\overline{Q}}^{\set,C^*}$ respectively on $T^d$. 
Hence, it follows for the lower bound
\begin{align*}
\widehat{\pi}\big(\underline{\widehat{Q}}^{\set,C^*}\big) 
&= \int_{\mathbb{R}^d} 
  \underline{\widehat{Q}}^{\set,C^*}(F_1(x_1),\dots,F_d(x_d))\ 
  \ud\mu_f(x_1,\dots,x_d)\\
& = \int_{\mathcal B} 
  \underline{\widehat{Q}}^{\set,C^*}(u_1,\dots,u_d)\ 
  \ud\mu_f(F_1^{-1}(u_1),\dots,F_d^{-1}(u_d))\\
& = \int_{\mathcal B} 
  \underline{\widehat{C}}^{\set,C^*}(u_1,\dots,u_d)\ 
  \ud\mu_f(F_1^{-1}(u_1),\dots,F_d^{-1}(u_d))
  = \widehat{\pi}_f\big(\underline{\widehat{C}}^{\set,C^*}\big),
\end{align*}
where we used the fact that $f(x_1,\dots,x_d)=V_f([0,x_1]\times\cdots\times[0,x_d])$ for the first equality and that $\operatorname{supp}\mu_f=\mathcal B$ for the second one. 
The third equality follows from $\underline{\widehat{Q}}^{S,Q^*}$ and $\underline{\widehat{C}}^{S,Q^*}$ being equal on $T^d$ and thus also on $\mathcal B$. 

In addition, using that $\underline{\widehat{C}}^{\set,C^*}$ is a copula that coincides with $\underline{\widehat{Q}}^{\set,C^*}$ on $T^d$ and $\underline{\widehat{Q}}^{\set,C^*}(\mathbf{x}) = \widehat{C}^*(\mathbf{x})$ for $\mathbf{x}\in\set\subset T^d$, it follows that $\underline{\widehat{C}}^{\set,C^*}\in\widehat{\mathcal{C}}^{\set,C^*}$, hence by the $\Delta$-monotonicity of $f$ we get that $\widehat{\pi}_f(\underline{\widehat{C}}^{\set,C^*}) = \inf\{\pi_f(C)\colon 
C\in\widehat{\mathcal{C}}^{\set,C^*}\}$. 
The proof for the upper bound can be obtained in the same way. 
\end{proof}

Finally, we are ready to apply our results in order to compute bounds on prices of multi-asset options when additional information on the dependence structure of $\mathbf{S}$ is available. 
The following examples illustrate this approach for different payoff functions and different kinds of additional information.

\begin{example}
Consider an option with payoff $f(\mathbf{S})$ on three assets $\mathbf{S}=(S_1,S_2,S_3)$. 
We are interested in computing bounds on the price of $f(\mathbf{S})$ assuming that partial information on the dependence structure of $\mathbf{S}$ is available. 
In particular, we assume that the marginal distributions $S_i\sim F_i$ are implied by the market prices of European call options. 
Moreover, we assume that partial information on the dependence structure stems from market prices of liquidly traded digital options of the form $\mathds{1}_{\max\{S_i,S_j\}<K}$ for $(i,j)=(1,2),(1,3),(2,3)$ and $K\in\mathbb{R}_+$. 
The prices of such options are immediately related to the copula $C$ of $\mathbf{S}$ since
$$
\mathbb{E}_{\mathbb{Q}}[\mathds{1}_{\max\{S_1,S_2\}<K}] 
  = \mathbb{Q}(S_1<K,S_2<K,S_3<\infty) 
  = C(F_1(K),F_2(K),1),
$$
and analogously for $(i,j)=(1,3),(2,3)$, for some martingale measure $\mathbb{Q}$.

Considering a set of strikes $\mathcal{K}:=\{K_1,\dots,K_n\}$, one can recover the values of the copula of $\mathbf{S}$ at several points. 
Let $\Pi^{(i,j)}_{K}$ denote the market price of a digital option on $(S_i,S_j)$ with strike $K$. 
These market prices imply then the following prescription on the copula of $\mathbf{S}$:
\begin{align}\label{eq:market-implied-prescription}
C(F_1(K),F_2(K),1) = \Pi^{(1,2)}_{K}, \nonumber\\
C(F_1(K),1,F_3(K)) = \Pi^{(1,3)}_{K}, \\\nonumber
C(1,F_2(K),F_3(K)) = \Pi^{(2,3)}_{K},
\end{align}
for $K\in\mathcal{K}$. 
Therefore, the collection of strikes induces a prescription on the copula on a compact subset of $\mathbb{I}^3$ of the form
\begin{align*}
\set = \bigcup_{K\in\mathcal{K}} 
  \big(F_1(K), F_2(K), 1\big) \cup 
  \big(F_1(K), 1, F_3(K)\big) \cup 
  \big(1, F_2(K), F_3(K)\big).
\end{align*}
The set of copulas that are compatible with this prescription is provided by
$$
\mathcal{C}^{\set,\Pi} 
  = \big\{C\in\mathcal{C}^3 \colon C(\mathbf{x}) = \Pi^{(i,j)}_{K} \ 
  \text{for all } \mathbf{x}\in \set\big\};
$$
see again \eqref{eq:market-implied-prescription}. 
Hence, we can now employ Theorem \ref{PrescriptionOnSubset} in order to compute the improved \FH bounds on the set $\mathcal{C}^{\set,\Pi}$ as follows:
\begin{align*}
\begin{split}
\overline{Q}^{\set,\Pi}(\mathbf{u}) 
  &= \min \Big( u_1,u_2,u_3, \min_{(i,j),K} \Big\{ \Pi^{(i,j)}_{K} 
  + \sum_{l=i,j} \big(u_l-F_l(K)\big)^+ \Big\} \Big) \\
\underline{Q}^{\set,\Pi}(\mathbf{u}) 
  &= \max \Big( 0, \sum_{i=1}^3 u_i-2, 
  \max_{\substack{(i,j),K,\\k\in\{1,2,3\}\setminus\{i,j\}}}
  \Big\{\Pi^{(i,j)}_{K}-\sum_{l=i,j} \big(F_l(K)-u_l\big)^+ 
    + (1-u_k)\Big\}\Big).
\end{split}
\end{align*}
Observe that the minimum and maximum in the equations above are taken over the set \set, using simply a more convenient parametrization. 
Using these improved \FH bounds, we can now apply Proposition \ref{noInformationBounds} and compute bounds on the price of an option with payoff $f(\mathbf{S})$ depending on all three assets. 
That is, we can compute bounds on the set of arbitrage-free option prices $\big\{\pi_f(C)\colon C \in \ \mathcal{C}^{\set,\Pi}\big\}$ which are compatible with the information stemming from pairwise digital options.

As an illustration of our results, we derive bounds on a digital option depending on all three assets, i.e. $f(\mathbf{S}) = \mathds{1}_{\max\{S_1,S_2,S_3\}<K}$.
In order to generate prices of pairwise digital options, we use the multivariate Black--Scholes model, therefore $\mathbf{S} = (S_1,S_2,S_3)$ is multivariate 
log-normally distributed with $S_i = s_i\exp(-\frac{1}{2} + X_i)$ where $(X_1,X_2,X_3)\sim\mathcal{N}(\mathbf 0,\Sigma)$ with 
$$\Sigma = 
\begin{pmatrix}
1 & \rho_{1,2} & \rho_{1,3} \\
\rho_{1,2} & 1 & \rho_{2,3}\\
\rho_{1,3} & \rho_{2,3} & 1
\end{pmatrix}.$$
Let us point out that this model is used to generate `traded' prices of pairwise digital options, but does not enter into the bounds.
The bounds are derived solely on the basis of the `traded' prices.

The following figures show the improved price bounds on the 3-asset digital option as a function of the strike $K$ as well as the price bounds using the `standard' \FH bounds, where we have fixed the initial values to $s_i=10$. 
As a benchmark, we also include the prices in the Black--Scholes model. 
We consider two scenarios for the pairwise correlations: in the left plot $\rho_{i,j}=0.3$ and in the right plot $\rho_{1,2}=0.5$, $\rho_{1,3}=-0.5$, $\rho_{2,3}=0$. 
\begin{figure}[H]
  \includegraphics[trim = 5mm 1mm 6mm 6mm, clip,width=0.5\textwidth,keepaspectratio=true]{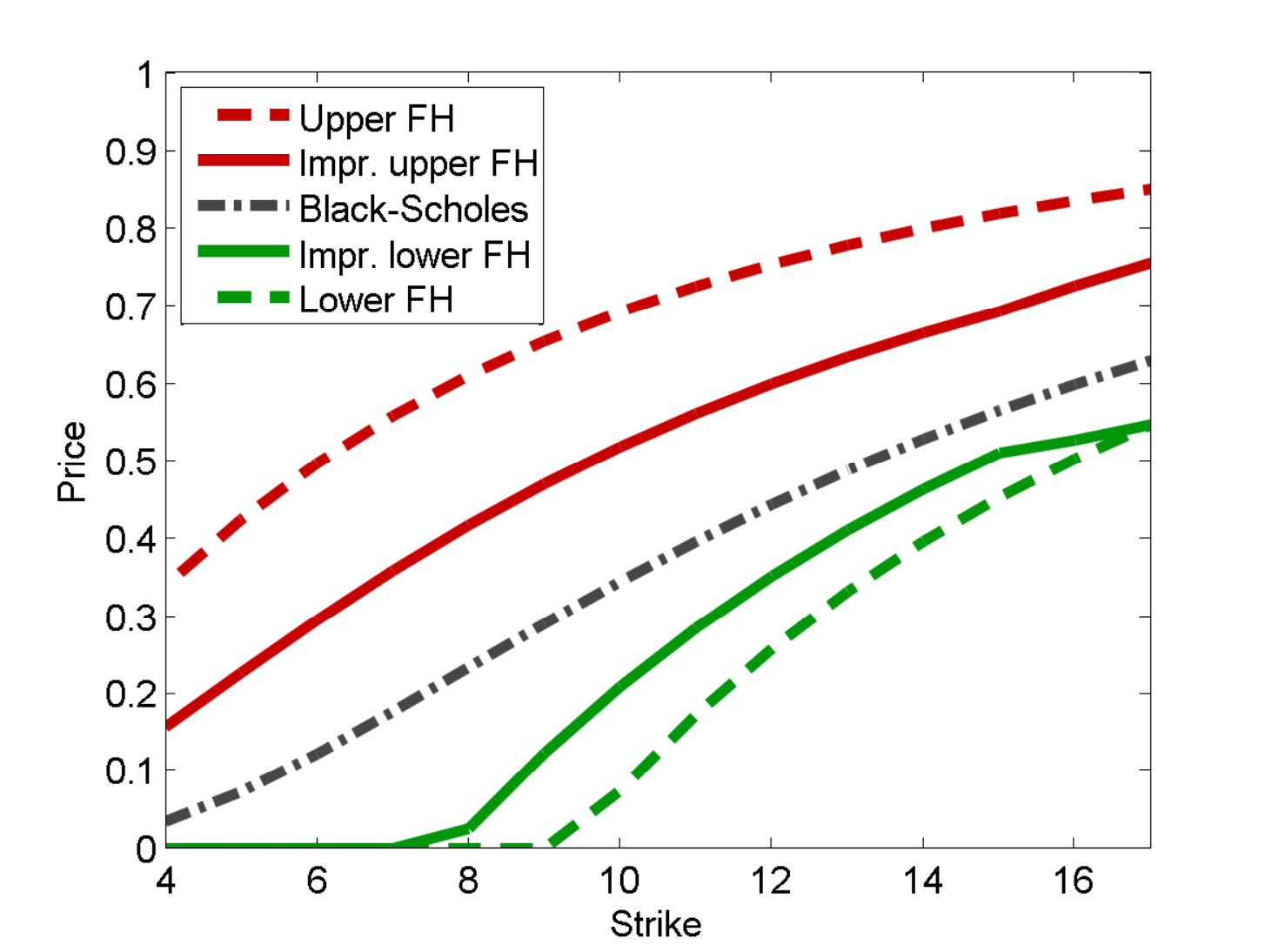}
  \includegraphics[trim = 5mm 1mm 6mm 6mm, clip,width=0.5\textwidth,keepaspectratio=true]{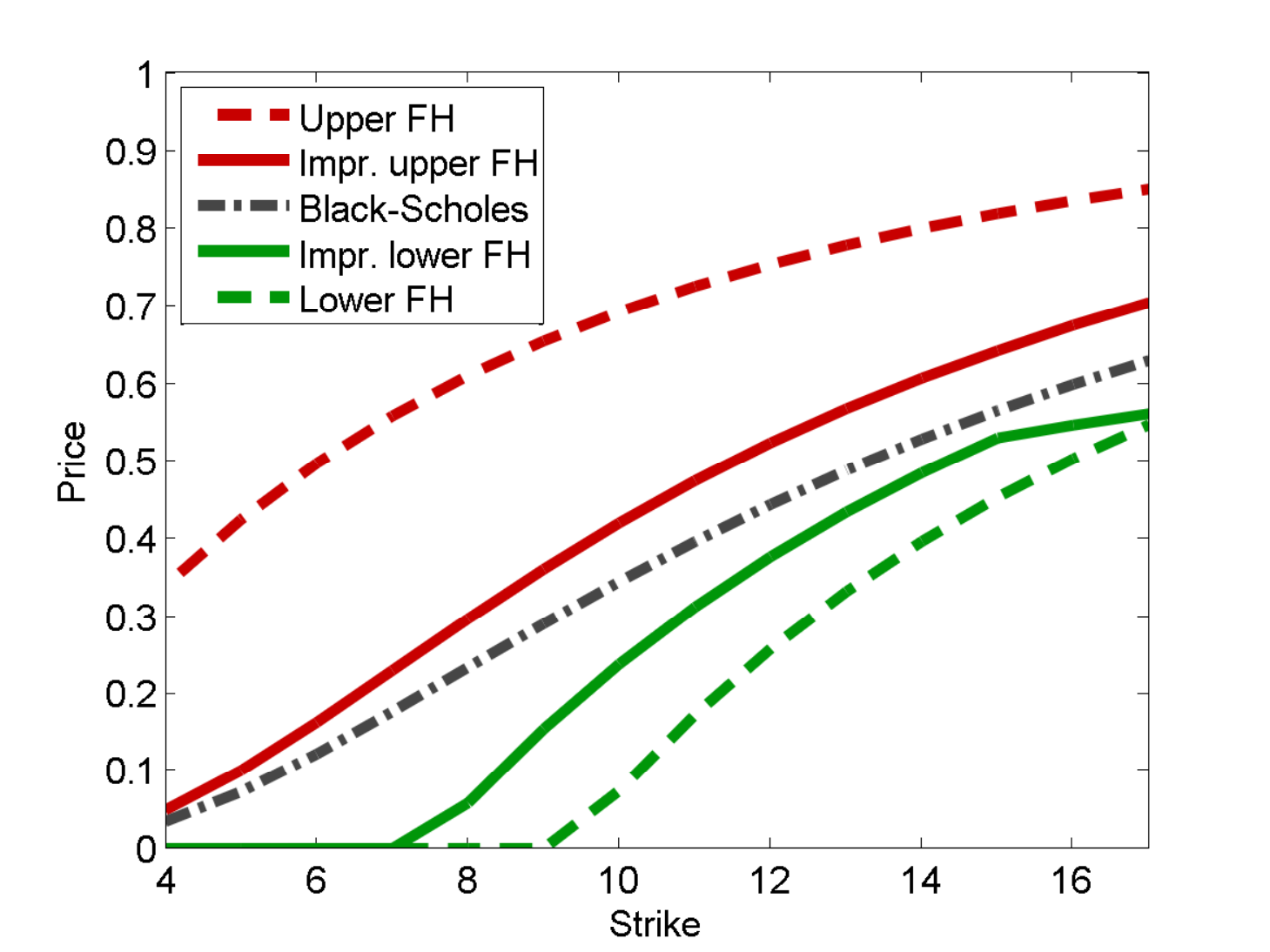}
  \vspace{-1.em}
\caption{Bounds on the prices of 3-asset digital options as functions of the strike.}
\end{figure}
Observe that the improved \FH bounds that account for the additional information from market prices of pairwise digital options lead in both cases to a considerable improvement of the option price bounds compared to the ones obtained with the `standard' \FH bounds. 
The improvement seems to be particularly pronounced if there are negative and positive correlations among the constituents of $\mathbf{S}$; see the right plot. 
\end{example}

\begin{example}
As a second example, we assume that digital options on $\mathbf{S} = (S_1,S_2,S_3)$ of the form $\mathds{1}_{\min\{S_1,S_2,S_3\}\geq K_i}$ for only two strikes $K_1,K_2\in\mathbb{R}_+$ are observed in the market. 
Their market prices are denoted by $\Pi_1,\Pi_2$, and immediately imply a prescription on the survival copula $\widehat{C}$ of $\mathbf{S}$ as follows:
$$
\Pi_i = \mathbb{Q}(S_1\geq K_i,S_2\geq K_i,S_3\geq K_i) 
	= \widehat{C}(F_1(K_i),F_2(K_i),F_3(K_i))
$$
for $i=1,2$. 
This is a prescription on two points, hence $\set=\{(F_1(K_i),F_2(K_i),F_3(K_i))\colon i=1,2\}\subset\mathbb{I}^3$, and we can employ Proposition \ref{PrescriptionOnSubsetSurvival} to compute the lower and upper bounds $\underline{\widehat{Q}}^{\set,\Pi}$ and $\widehat{\overline{Q}}^{\set,\Pi}$ on the set of copulas $\widehat{\mathcal{C}}^{\set,\Pi}=\{C \in \mathcal{C}^3 \colon \widehat{C}(\mathbf{x}) = \Pi_i,\ \mathbf{x} \in\set\}$ which are compatible with this prescription. 
We have that
\begin{align*}
\begin{split}
\widehat{\overline{Q}}^{\set,\Pi}(\mathbf{u}) 
	&= \min\Big(1-u_1,1-u_2,1-u_3,\min_{i=1,2}\Big\{\Pi_i+\sum_{l=1}^3 
		\big(F_l(K_i)-u_l\big)^+\Big\}\Big), \\
\underline{\widehat{Q}}^{\set,\Pi}(\mathbf{u}) 
	&= \max\Big(0, \sum_{i=1}^3 (1-u_i)-2,\max_{i=1,2} 
		\Big\{\Pi_i-\sum_{l=1}^3 \big(u_l-F_l(K_i)\big)^+\Big\}\Big).
\end{split}
\end{align*}

\begin{figure}[H]
  \includegraphics[trim = 5mm 1mm 6mm 6mm, clip,width=0.5\textwidth,keepaspectratio=true]{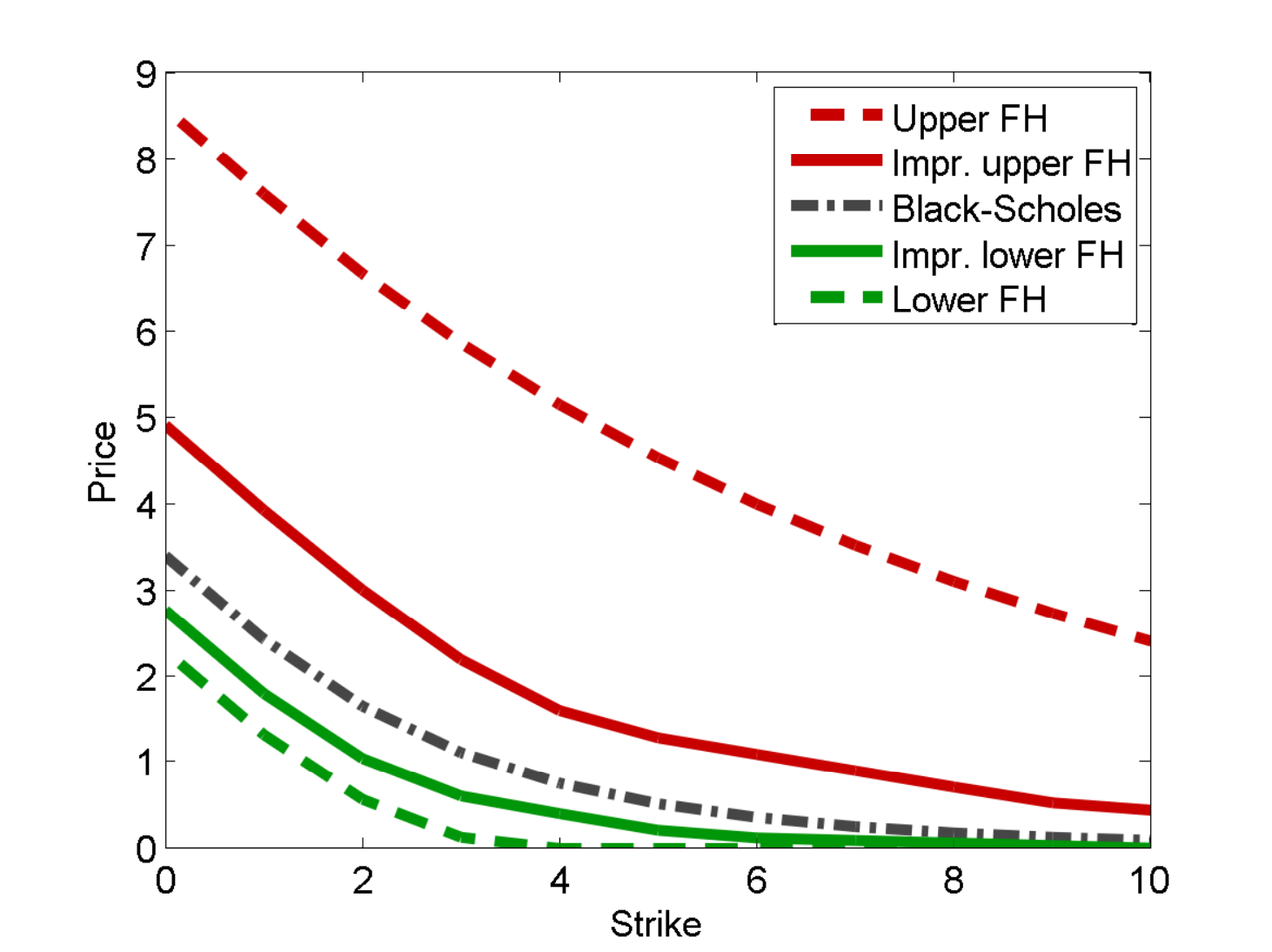}
  \includegraphics[trim = 5mm 1mm 6mm 6mm, clip,width=0.5\textwidth,keepaspectratio=true]{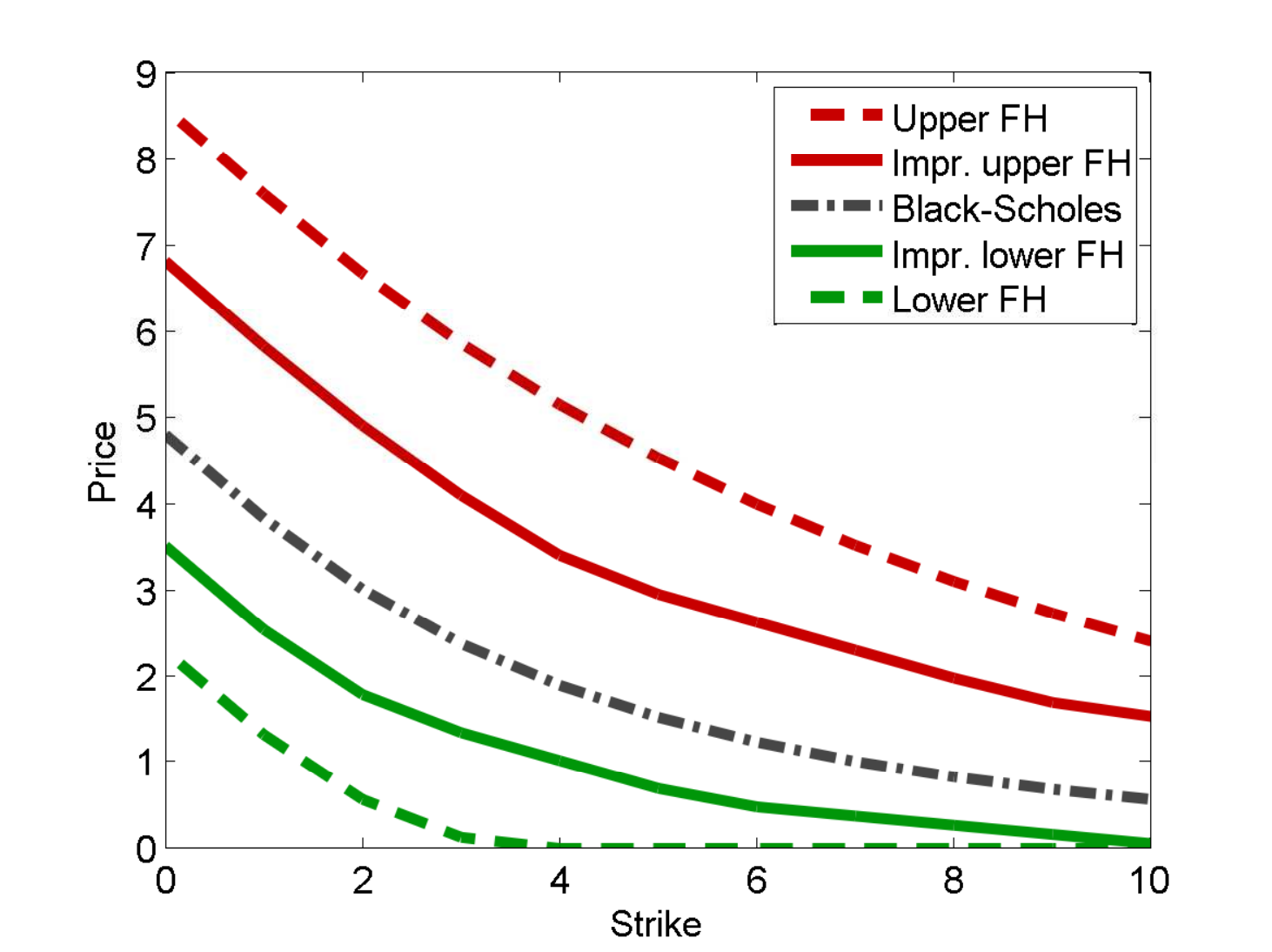}
  \vspace{-1.em}
\caption{\label{fig:2}Bounds on the prices of options on the minimum of $\mathbf{S}$ as functions of the strike.}
\end{figure}

Using these bounds, we can now apply Proposition \ref{noInformationBounds} and compute improved bounds on the set of arbitrage-free prices for a call option on 
the minimum of $\mathbf{S}$, whose payoff is $f(\mathbf{S}) = (\min\{S_1,S_2,S_3\}-K)^+$. 
The set of prices for $f(\mathbf{S})$ that are compatible with the market prices of given digital options is denoted by $\Pi^* = \big\{\widehat{\pi}_f(\widehat{C})\colon C \in \ \widehat{\mathcal{C}}^{\set,\Pi}\big\}$ and, since $f$ is $\Delta$-monotonic, it holds that $\widehat{\pi}_f(\underline{\widehat{Q}}^{\set,\Pi})\leq \pi \leq \widehat{\pi}_f(\widehat{\overline{Q}}^{\set,\Pi})$ for all $\pi\in\Pi^*$. 
The computation of $\widehat{\pi}_f(Q)$ reduces to
$$
\widehat{\pi}_f(Q) = \int_K^\infty Q\big(F_1(x),F_2(x),F_3(x)\big) \ud x,
$$
see Table \ref{tab:payoff-rep}, which is an integral over a subset of the 3-track 
$$
\{ (F_1(x),F_2(x),F_3(x)) \colon x\in\overline{\mathbb{R}}_+ \} 
  \supset \{ (F_1(x),F_2(x),F_3(x)) \colon x\in[K,\infty) \}
  \supset\set.
$$ 
Hence, Lemma \ref{sharpBounds} yields that the price bounds $\widehat{\pi}_f(\underline{\widehat{Q}}^{\set,\Pi})$ and $\widehat{\pi}_f(\widehat{\overline{Q}}^{\set,\Pi})$ are sharp, that is
$$
\widehat{\pi}_f\big(\underline{\widehat{Q}}^{\set,\Pi}\big) 
	= \inf\{\pi\colon \pi\in\Pi^*\}
\quad\text{ and }\quad
\widehat{\pi}_f\big(\widehat{\overline{Q}}^{\set,\Pi}\big) 
	= \sup\{\pi\colon \pi\in\Pi^*\}.
$$

Analogously to the previous example we assume, for the sake of a numerical illustration, that $\mathbf{S}$ follows the multivariate Black--Scholes model and the pairwise correlations are denoted by $\rho_{i,j}$. 
The parameters of the model remain the same as in the previous example. 
We then use this model to generate prices of digital options that determine the prescription. 
Figure \ref{fig:2} depicts the bounds on the prices of a call on the minimum of $\mathbf{S}$ stemming from the improved \FH bounds as a function of the strike 
$K$, as well as those from the `standard' \FH bounds. 
The price from the multivariate Black--Scholes model is also included as a benchmark. 
Again we consider two scenarios for the pairwise correlations: in the left plot $\rho_{i,j}=0$ and in the right one $\rho_{i,j}=0.5$. 
We can observe once again, that the use of the additional information leads to a significant improvement of the bounds relative to the `standard' situation, although in this example the additional information is just two prices.
\end{example}

\section{Conclusion}

This paper provides upper and lower bounds on the expectation of $f(\mathbf S)$ where $f$ is a function and $\mathbf S$ is a random vector with known marginal distributions and partially unknown dependence structure.
The partial information on the dependence structure can be incorporated via improved \FH bounds that take this into account.
These bounds are typically quasi-copulas, and not (proper) copulas.
Therefore, we provide an alternative representation of multivariate integrals with respect to copulas that allows for quasi-copulas as integrators, and new integral characterizations of orthant orders on the set of quasi-copulas.
As an application of these results, we derive model-free bounds on the prices of multi-asset options when partial information on the dependence structure between the assets is available.
Numerical results demonstrate the improved performance of the price bounds that utilize the improved \FH bounds on copulas.

\appendix

\section{Improved \FH bounds for survival copulas}
\label{AppA}

In this section we establish improved \FH bounds in the presence of additional information for survival copulas, analogous to those derived in Section \ref{sec:iFH} for copulas. 
The first Proposition establishes improved bounds if the survival copula is prescribed on a compact set.

\begin{proposition}
\label{PrescriptionOnSubsetSurvival}
Let $\set\subset\mathbb{I}^d$ be a compact set and $\widehat{Q}^*$ be a quasi-survival function. Consider 
the set 
\begin{align*}
\mathcal{\widehat{C}}^{\set,\widehat{Q}^*} 
  := \big\{ C\in\mathcal{C}^d\colon \widehat{C}(\mathbf{x}) 
  = \widehat{Q}^*(\mathbf{x}) \text{ for all } \mathbf{x}\in \set\big\}.
\end{align*}
Then, for all $C\in\mathcal{\widehat{C}}^{\set,\widehat{Q}^*}$, it holds that
\begin{align}
\begin{split}
&\underline{\widehat{Q}}^{\set,\widehat{Q}^*}(\mathbf{u}) \leq 
  \widehat{C}(\mathbf{u}) \leq \widehat{\overline{Q}}^{\set,\widehat{Q}^*}(\mathbf{u})
  \quad\text{for all }\mathbf{u}\in\mathbb{I}^d,\\
&\underline{\widehat{Q}}^{\set,\widehat{Q}^*}(\mathbf{u}) = 
  \widehat{C}(\mathbf{u}) = \widehat{\overline{Q}}^{\set,\widehat{Q}^*}(\mathbf{u})
  \quad\text{for all }\mathbf{u}\in \set,
\end{split}
\end{align}
where the bounds are provided by
$$\underline{\widehat{Q}}^{\set,\widehat{Q}^*}(\mathbf{u}) 
  :=  \underline{Q}^{\widehat{\set},\widehat{Q}^*}(1-u_1,\dots,1-u_d) 
\quad\text{and}\quad 
\widehat{\overline{Q}}^{\set,\widehat{Q}^*}(\mathbf{u}) 
  := \overline{Q}^{\widehat{\set},\widehat{Q}^*}(1-u_1,\dots,1-u_d)$$
with $\widehat{\set} = \{(1-x_1,\dots,1-x_d)\colon (x_1,\dots,x_d)\in \set\}$.
\end{proposition}

\begin{proof}
Let $C\in\mathcal{\widehat{C}}^{\set,\widehat{Q}^*}$. Since $C$ is a copula we know that 
$\widehat{C}(1-u_1,\dots,1-u_d)$ is also a copula. Defining $v_i:=1-x_i$, the 
prescription $\widehat{C}(1-v_1,\dots,1-v_d) = \widehat{Q}^*(x_1,\dots,x_d)$ 
holds for all $\bx\in \set$ by assumption. Thus by Theorem 
\ref{PrescriptionOnSubset} we obtain
$$ \underline{Q}^{\widehat{\set},\widehat{Q}^*}(u_1,\dots,u_d) 
  \leq \widehat{C}(1-u_1,\dots,1-u_d) 
  \leq \overline{Q}^{\widehat{\set},\widehat{Q}^*}(u_1,\dots,u_d)$$
which by a transformation of variables equals 
\begin{align*}
\underline{Q}^{\widehat{\set},\widehat{Q}^*}(1-u_1,\dots,1-u_d) 
  \leq \widehat{C}(u_1,\dots,u_d) 
  \leq \overline{Q}^{\widehat{\set},\widehat{Q}^*}(1-u_1,\dots,1-u_d). 
\end{align*}
\end{proof}

The next result establishes improved bounds if the value of a functional which is increasing with respect to the upper orthant order is prescribed. 
The proof is analogous to the proof of Theorem \ref{prescribedFunctional} and is therefore omitted. 

\begin{proposition}
\label{prescribedFunctionalSurvival}
Let $\mathrm C(\mathbb I^d)$ denote the set of continuous functions on $\mathbb I^d$, $\rho\colon \mathrm C(\mathbb I^d)\to\mathbb{R}$ be increasing with respect to the upper orthant order on $\mathcal{C}^d$ and continuous with respect to the pointwise convergence of copulas. 
Define
$$\widehat{\mathcal{C}}^{\rho,\theta}:=\{C\in \mathrm C(\mathbb I^d)\colon \rho(\widehat{C})=\theta\}.$$
Then for all $C\in\widehat{\mathcal{C}}^{\rho,\theta}$ it holds
$$\underline{\widehat{Q}}^{\rho,\theta}(\mathbf{u}) 
  \leq \widehat{C}(\mathbf{u}) 
  \leq \widehat{\overline{Q}}^{\rho,\theta}(\mathbf{u})
  \quad \text{for all }\mathbf{u}\in\mathbb{I}^d,$$
where the bounds are provided by
\begin{align*}
\underline{\widehat{Q}}^{\rho,\theta}(\mathbf{u}) 
:= \begin{cases}
  \rho_+^{-1}(\mathbf{u},\theta),
  & 
\theta\in\Big[\rho\Big(\widehat{\overline{Q}}^{\{\mathbf{u}\},W_d(\mathbf 1 - \mathbf{u})}\Big),
    \rho(M_d(\mathbf 1-\cdot))\Big]\\
  W_d(\mathbf{1}-\mathbf{u}), & \text{else},
\end{cases}
\end{align*}
and
\begin{align*}
\widehat{\overline{Q}}^{\rho,\theta}(\mathbf{u})
:= \begin{cases}
  \rho_-^{-1}(\mathbf{u},\theta),
  &\theta\in\Big[\rho(W_d(\mathbf{1}-\cdot)),\rho\Big(\underline{\widehat{Q}}^{\{\mathbf{u}\}, 
    M_d(\mathbf 1 - \mathbf{u})}\Big)\Big]\\
  M_d(\mathbf{1}-\mathbf{u}), & \text{else},
\end{cases}
\end{align*}
where
\begin{align*}
\rho_-^{-1}(\mathbf{u},\theta) 
  = \max\Big\{r\colon \rho\big(\underline{\widehat{Q}}^{\{\mathbf{u}\},r}\big) 
    = \theta\Big\}\qquad
\text{and}\qquad 
\rho_+^{-1}(\mathbf{u},\theta) 
  = \min\Big\{r\colon \rho\big(\widehat{\overline{Q}}^{\{\mathbf{u}\},r}\big) 
    = \theta\Big\},
\end{align*}
while the quasi-copulas $\underline{\widehat{Q}}^{\{\mathbf{u}\},r}$ and 
$\widehat{\overline{Q}}^{\{\mathbf{u}\},r}$ are given in Proposition 
\ref{PrescriptionOnSubsetSurvival} for $r\in\mathbb{I}$.
\end{proposition}

Finally, the subsequent Proposition is a version of Theorem \ref{regionalPrescription} for survival copulas. 
Its proof is analogous to the proof of Theorem \ref{regionalPrescription} and is therefore also omitted. 
Recall the definitions of the projection and lift operations on a vector and the definition of the $I$-margin of a copula. 
Moreover, recall that $\widehat{Q_I}$ denotes the survival function of $Q_I$. 

\begin{proposition}
\label{regionalPrescriptionSurvival}
Let $I_1,\dots,I_k$ be subsets of $\{1,\dots,d\}$ with $|I_j|\geq 2$ for 
$j\in\{1,\dots,k\}$ and $|I_i\cap I_j| \leq 1$ for $i,j\in\{1,\dots,k\}$, 
$i\neq j$. Let $\underline{Q}_j,\overline{Q}_j$ be $|I_j|$-quasi-copulas with 
$\underline{Q}_j\preceq_{UO}\overline{Q}_j$ for $j=1,\dots,k$ and consider the 
set
$$\widehat{\mathcal{C}}^I
  = \left\{C\in\mathcal{C}^d \colon \underline{Q}_j\preceq_{UO} 
  C_{I_j}\preceq_{UO} \overline{Q}_j,\ j=1,\dots,k \right\},$$
where $C_{I_j}$ is the $I_j$-margin of $C$. Then it holds for all 
$C\in\widehat{\mathcal{C}}^I$
$$\widehat{\overline{Q}}^I(\mathbf{u}) 
  \leq \widehat{C}(\mathbf{u}) 
  \leq \underline{\widehat{Q}}^I(\mathbf{u}) 
  \quad \text{for all }u\in\mathbb{I}^d,$$
where
\begin{align*}
&\widehat{\overline{Q}}^I(u_1,\dots,u_d) := \overline{Q}^I(1-u_1,\dots,1-u_d),\\
&\underline{\widehat{Q}}^I(u_1,\dots,u_d) := \underline{Q}^I(1-u_1,\dots,1-u_d),
\end{align*} 
while $\underline{Q}^I,\overline{Q}^I$ are provided by Theorem 
\ref{regionalPrescription}.
\end{proposition}

\section{The improved \FH bounds are not copulas: the general case}
\label{AppB}

The following Theorem is a generalization of Theorem \ref{lowerBoundIsQuasiCopula} for $d>3$.
\begin{theorem}
\label{lowerBoundIsQuasiCopulaApp}
Consider the compact subset $\set$ of $\I^d$
\begin{align}\label{eq:app-set}
\set &= [0,1]\times\cdots\times[0,1] \!\times\!
    \underset{i-\text{th component}}
      {\underbrace{\Big([0,1]\setminus(s_i,s_i+\varepsilon_i)\Big)}} 
    \times[0,1]\times \!\cdots\! \times[0,1]\times
    \underset{j-\text{th component}} 
      {\underbrace{\Big([0,1]\setminus(s_j,s_j+\varepsilon_j)\Big)}} \notag\\
&\quad \times[0,1]\times \!\cdots\! \times[0,1]\times
    \underset{k-\text{th component}} 
      {\underbrace{\Big([0,1]\setminus(s_k,s_k+\varepsilon_k)\Big)}}
    \times[0,1] \!\cdots\! \times[0,1],
\end{align}
for $\mathbf{s}=(s_i,s_j,s_k),\overline{\mathbf{s}}=(s_i+\varepsilon_i,s_j+\varepsilon_j,s_k+\varepsilon_k)\in\mathbb{I}^3$ and 
$\varepsilon_i,\varepsilon_j,\varepsilon_k>0$. Moreover, let $C^*$ be a 
$d$-copula (or a $d$-quasi-copula) such that
\begin{align}
& \sum_{i=1}^3 \varepsilon_i > 
  C^*(\overline{\mathbf{s}}'_I)-C^*(\mathbf{s}'_I)>0, \label{reqApp1.1}\\
& C^*(\mathbf{s}'_I) \geq 
  W_d(\overline{\mathbf{s}}'_I) \label{reqApp1.2},
\end{align}
where $I:=\{i,j,k\}$ and $\mathbf{s}'_I$, $\overline{\mathbf{s}}'_I$ are defined
by the lift operation. Then $\underline{Q}^{\set,C^*}$ is a proper quasi-copula.
\end{theorem}

\begin{proof}
From Theorem \ref{lowerBoundIsQuasiCopula} we know already that the statement holds if $d=3$. For the general case, i.e. $d>3$, choose $u_l\in[0,1]$ with $u_l\in(s_l,s_l+\varepsilon_l)$ for $l\in I=\{i,j,k\}$, such that
\begin{align*}
& C^*(\overline{\mathbf{s}}'_I)-C^*(\mathbf{s}'_I) 
  < \sum_{l\in I} (s_l+\varepsilon_l -u_l) \quad \text{and} \\
& C^*(\overline{\mathbf{s}}'_I)-C^*(\mathbf{s}'_I) 
  > \sum_{l\in J}(s_l+\varepsilon_l -u_l)\quad\text{for } J = (i,j),(j,k),(i,k);
\end{align*}
this exists due to \eqref{reqApp1.1}. Then, considering the set 
\begin{align*}
H &= [0,1]\times\cdots\times[0,1]\times[u_i,s_i+\varepsilon_i]\times[0,1]
  \times\cdots\times[0,1]\times[u_j,s_j+\varepsilon_j]\times[0,1]\times\cdots\\
&\quad \times[0,1]\times[u_j,s_j+\varepsilon_j]\times[0,1]\times\dots\times[0,1]
\end{align*}
and using similar argumentation as in the case $d=3$ together with property  \ref{cond:QC1}, it follows that
\begin{align*}
V_{\underline{Q}^{\set,C^*}}(H) 
&= \underline{Q}^{\set,C^*} \big(\overline{\mathbf{s}}'\big) 
 - \underline{Q}^{\set,C^*} \big((u_i,s_j+\varepsilon_j,s_k+\varepsilon_k)'\big) 
 - \dots \\
&\quad +\underline{Q}^{\set,C^*} \big((u_i,u_j,s_k+\varepsilon_k)'\big) + \dots
 - \underline{Q}^{\set,C^*}\big(\bu_I'\big) \\
&= C^*\big(\overline{\mathbf{s}}'_I\big)
 - 3\, C^*\big(\overline{\mathbf{s}}'_I\big) 
 + \sum_{l\in I}(s_l+\varepsilon_l-u_l) \\
&\quad + 3\, C^*\big(\overline{\mathbf{s}}'_I\big) 
 - 2\sum_{l\in I} (s_l+\varepsilon_l-u_l) - C^*\big(\mathbf{s}'_I\big)\\
&= C^*\big(\overline{\mathbf{s}}'_I\big) - C^*\big(\mathbf{s}'_I\big)
 - \sum_{l\in I} (s_l+\varepsilon_l-u_l)<0.
\end{align*}
Hence, $\underline{Q}^{\set,C^*}$ is a proper-quasi-copula.
\end{proof}

A general version of Corollary \ref{lowerBoundIsQuasiCopulaCor} also holds.
\begin{corollary}
Let $C^*$ be a $d$-copula and $\set\subset\mathbb{I}^d$ be compact. If there 
exists a compact set $\set'\supset \set$ such that $\set'$ and 
$Q^*:=\underline{Q}^{\set,C^*}$ satisfy the assumptions of Theorem 
\ref{lowerBoundIsQuasiCopulaApp}, then $\underline{Q}^{\set,C^*}$ is a proper 
quasi-copula.
\end{corollary}

The next result establishes similar conditions for the upper bound to be a proper-quasi copula. 

\begin{theorem}
\label{upperBoundIsQuasiCopulaApp}
Consider the compact subset $\set$ of $\I^d$ in \eqref{eq:app-set} for $\mathbf{s}=(s_i,s_j,s_k),\overline{\mathbf{s}}=(s_i+\varepsilon_i,s_j+\varepsilon_j,s_k+\varepsilon_k)\in\mathbb{I}^3$ and $\varepsilon_i,\varepsilon_j,\varepsilon_k>0$. Let $C^*$ be a $d$-copula (or $d$-quasi-copula) such that 
\begin{align}
& \sum_{i=1}^3 \varepsilon_i>C^*(\overline{\mathbf{s}}'_I)-C^*(\mathbf{s}'_I)>0,  \label{reqApp1.5}\\
& C^*(\overline{\mathbf{s}}'_I)\leq M_d(\mathbf{s}'_I),
\label{reqApp1.6}
\end{align}
where $I=\{i,j,k\}$ and $\mathbf{s}'_I$ and $\overline{\mathbf{s}}'_I$ are defined by the lift operation. Then $\overline{Q}^{\set,C^*}$ is a proper quasi-copula.

\begin{proof}
We show that the result holds for $d=3$. The general case for $d>3$ follows as in the proof of Theorem \ref{lowerBoundIsQuasiCopulaApp}. Let $C^*$ be a 3-copula and $\set=\mathbb{I}^3\setminus(\mathbf{s},\mathbf{s}+\pmb{\varepsilon})$ for some $\mathbf{s}\in[0,1]^3$ and $\varepsilon_i>0$, $i=1,2,3$. Moreover, choose $\mathbf{u}=(u_1,u_2,u_3)\in(\mathbf{s},\mathbf{s}+\pmb{\varepsilon})$ such that
\begin{align}
&C^*(\mathbf{s}+\pmb{\varepsilon})-C^*(\mathbf{s}) < \sum_{i=1}^3 
(s_i+\varepsilon_i -u_i)\quad\text{and}\label{reqApp1.7}\\
&C^*(\mathbf{s}+\pmb{\varepsilon})-C^*(\mathbf{s}) > \sum_{i\in I} 
(s_i+\varepsilon_i -u_i)\quad \text{for } I = (1,2),(2,3),(1,3); \label{reqApp1.8}
\end{align}
such a $\mathbf{u}$ exists due to \eqref{reqApp1.5}. Now, in order to show that $\overline{Q}^{\set,C^*}$ is not $d$-increasing, and thus a proper quasi-copula, it suffices to prove that $V_{\overline{Q}^{\set,C^*}}([\mathbf{s},\mathbf{u}])<0$. By the definition of $V_{\overline{Q}^{\set,C^*}}$ we have
\begin{align*}
V_{\overline{Q}^{\set,C^*}}([\mathbf{s},\mathbf{u}]) 
&= \overline{Q}^{\set,C^*}(\mathbf{u})-\overline{Q}^{\set,C^*}(s_1,u_2,u_3)-\overline{Q}^{\set,C^*}(u_1,s_2,u_3)-\overline{Q}^{\set,C^*}(u_1,u_2,s_3)\\
&\quad +\overline{Q}^{\set,C^*}(s_1,s_2,u_3)+\overline{Q}^{\set,C^*}(s_1,u_2,s_3)+\overline{Q}^{\set,C^*}(u_1,s_2,s_3) - \overline{Q}^{\set,C^*}(\mathbf{s}).
\end{align*}

Analyzing the summands we see that
\begin{itemize}
\item $\overline{Q}^{\set,C^*}(\mathbf{u}) = \min_{x\in \set}\Big\{C^*(\mathbf{x})+\sum_{i=1}^3(x_i-u_i)^+\Big\} =  
C^*(\mathbf{s}+\pmb{\varepsilon})$, where the first equality holds due to \eqref{reqApp1.6} and the second one due to \eqref{reqApp1.7}.

\item $\overline{Q}^{\set,C^*}(s_1,u_2,u_3) =  \min_{x\in \set}\Big\{C^*(\mathbf{x})+(s_1-x_1)^++(u_2-x_2)^++(u_3-x_3)^+\Big\} =  \min_{x\in 
\set}\Big\{C^*(\mathbf{s}+\pmb{\varepsilon}), C^*(\mathbf{s})+(u_2-x_2)^++(u_3-x_3)^+\Big\} =  C^*(\mathbf{s})+(u_2-s_2)+(u_3-s_3)$, where the first equality holds due to \eqref{reqApp1.6} and the third one due to \eqref{reqApp1.8}. The second equality holds since the minimum is only attained at either $\mathbf{s}$ or $\mathbf{s}+\pmb{\varepsilon}$. Analogously it follows that $\overline{Q}^{\set,C^*}(u_1,s_2,u_3) = C^*(\mathbf{s})+(u_1-s_1)+(u_3-s_3)$ and $\overline{Q}^{\set,C^*}(u_1,u_2,s_3) = C^*(\mathbf{s})+(u_1-s_1)+(u_2-s_2)$.

\item Using similar argumentation it follows that $\overline{Q}^{\set,C^*}(s_1,s_2,u_3) = C^*(\mathbf{s})+(u_3-s_3)$, $\overline{Q}^{\set,C^*}(u_1,s_2,s_3) = C^*(\mathbf{s})+(u_1-s_1)$ and $\overline{Q}^{\set,C^*}(s_1,u_2,s_3) = C^*(\mathbf{s})+(u_2-s_2)$.

\item In addition, $\overline{Q}^{\set,C^*}(\mathbf{s}) = C^*(\mathbf{s})$ because $s\in \set$.
\end{itemize}

Therefore, putting the pieces together and using \eqref{reqApp1.7}, we get
\begin{align*}
V_{\underline{Q}^{\set,C^*}}([\mathbf{s},\mathbf{u}]) & = 
C^*(\mathbf{s}+\pmb{\varepsilon})-3C^*(\mathbf{s}
)-2\sum_{i=1}^3 (u_i-s_i) + 3C^*(\mathbf{s}
)+\sum_{i=1}^3 (u_i-s_i)-C^*(\mathbf{s})\\
&=C^*(\mathbf{s}+\pmb{\varepsilon})-C^*(\mathbf{s}) - \sum_{i=1}^3 (u_i-s_i)<0.
\end{align*}
Thus $\underline{Q}^{\set,C^*}$ is indeed a proper quasi-copula. 
\end{proof}
\end{theorem}

The following Corollary shows that the requirements in Theorem 
\ref{upperBoundIsQuasiCopulaApp} are minimal in the sense that if the 
prescription set is contained in a set of the form \eqref{eq:app-set} then the 
upper bound is a proper-quasi-copula. Its proof is analogous to the proof of 
Corollary \ref{lowerBoundIsQuasiCopulaCor} and therefore omitted.

\begin{corollary}
\label{upperBoundIsQuasiCopulaCor}
Let $C^*$ be a $d$-copula and $S\subset\mathbb{I}^d$ be compact. If there exists 
a compact set $\set'\supset \set$ such that $\set'$ and 
$Q^*:=\overline{Q}^{\set,C^*}$ satisfy the assumptions of Theorem 
\ref{upperBoundIsQuasiCopulaApp}, then $\overline{Q}^{\set,C^*}$ is a proper 
quasi-copula.
\end{corollary}

\bibliographystyle{abbrvnat}
\bibliography{bib}

\end{document}

%% file: head.tex
\usepackage{times}
\usepackage{calc}
\usepackage{enumitem}
\usepackage{eurosym}
\usepackage{textcomp}
\usepackage{mathrsfs}
\usepackage{graphicx}
\usepackage{enumitem}
\usepackage{stmaryrd}
\usepackage{epsfig}
\usepackage[usenames,dvipsnames]{xcolor}

\newcommand{\set}[1]{\ensuremath{ \{ #1 \} }}
\newcommand{\R}{\mathbb{R}}

\usepackage[hyperref,amsmath,thmmarks]{ntheorem}

\theoremseparator{.}

\theoremsymbol{\ensuremath{\blacklozenge}}
\theoremheaderfont{\itshape}

\theoremsymbol{\ensuremath{\square}}
\theoremheaderfont{\itshape}
\theoremstyle{nonumberplain}

\numberwithin{equation}{section}